\documentclass[3p]{elsarticle}

\usepackage[english]{babel}
\usepackage{amsmath, amsfonts, amssymb, amsthm}
\usepackage{enumitem}
\usepackage[usenames, dvipsnames]{color}
\usepackage{tikz}

\def\Bscr{\mathcal{B}}
\def\Dscr{\mathcal{D}}
\def\Gscr{\mathcal{G}}
\def\Iscr{\mathcal{I}}
\def\Mscr{\mathcal{M}}
\def\Oscr{\mathcal{O}}
\def\Pscr{\mathcal{P}}
\def\Uscr{\mathcal{U}}
\def\Vscr{\mathcal{V}}
\newcommand{\tond}[1]{{\left(#1\right)}}
\newcommand{\quadr}[1]{{\left[#1\right]}}
\newcommand{\inter}[1]{{\langle#1\rangle}}
\newcommand{\graff}[1]{{\left\{#1\right\}}}

\newcommand{\norm}[1]{\left\|#1\right\|}
\newcommand{\opnorm}[1]{\left\|#1\right\|_{\textsf{op}}}
\newcommand{\matext}[1]{#1_{\#}}
\newcommand{\matred}[1]{#1_{\flat}}

\DeclareMathOperator{\dist}{dist}

\def\naturali{\mathbb{N}}
\def\interi{\mathbb{Z}}
\def\reali{{\mathbb{R}}}
\def\complessi{{\mathbb{C}}}
\def\Mat{{\textsf{Mat}}}

\def\toro{\mathbb{T}}
\def\Id{{\rm Id}}
\def\rmI{{\rm I}}
\def\rmII{{\rm II}}
\def\rmIII{{\rm III}}
\def\rmIV{{\rm IV}}

\def\epsilon{\varepsilon}
\def\imunit{{\bf i}}

\def\lie#1{L_{#1}}
\def\Pset{{\cal P}}
\def\realpart{\mathop{\rm Re}\nolimits}
\def\imaginary{\mathop{\rm Im}\nolimits}
\def\prsca#1#2{\langle #1 , #2 \rangle}
\def\parder#1#2{{\partial#1 \over \partial#2}}

\newtheorem{theorem}{Theorem}[section]
\newtheorem{theorem*}{Theorem}
\newtheorem{lemma}{Lemma}[section]
\newtheorem{corollary}{Corollary}[section]
\newtheorem{proposition}{Proposition}[section]

\newtheorem{remark}{Remark}[section]

\title{On the continuation of degenerate periodic orbits\\ via normal form:
       lower dimensional resonant tori}

\author[1]{M. Sansottera}
\author[1]{V. Danesi}
\author[1]{T. Penati}
\author[1]{S. Paleari}

\address[1]{Department of Mathematics, University of Milan, via
  Saldini 50, 20133 --- Milan, Italy.}

\begin{document}

\begin{abstract}
We consider the classical problem of the continuation of periodic
orbits surviving to the breaking of invariant lower dimensional
resonant tori in nearly integrable Hamiltonian systems.  In particular
we extend our previous results (presented in CNSNS, 61:198-224, 2018)
for full dimensional resonant tori to lower dimensional ones.  We
develop a constructive normal form scheme that allows to identify and
approximate the periodic orbits which continue to exist after the
breaking of the resonant torus.  A specific feature of our algorithm
consists in the possibility of dealing with degenerate periodic
orbits. Besides, under suitable hypothesis on the spectrum of the
approximate periodic orbit, we obtain information on the linear
stability of the periodic orbits feasible of continuation. A
pedagogical example involving few degrees of freedom, but connected to
the classical topic of discrete solitons in dNLS-lattices, is also
provided.
\end{abstract}


\maketitle

\section{Introduction}
\label{sec:introduction}

Consider a canonical system of differential equations with Hamiltonian
\begin{equation}
H(I,\varphi,\xi,\eta,\epsilon) = H_0(I,\xi,\eta) +\epsilon H_1(I,\varphi,\xi,\eta;\epsilon)\ ,\quad
H_0=h_0(I)+g_0 (\xi,\eta)\ ,
\label{frm:H-modello}
\end{equation}
in $n=n_1+n_2$ degree of freedom where
$(I,\varphi)\in\Uscr(I^*)\times\toro^{n_1}$ are angle-action variables
defined in a neighbourhood $\Uscr(I^*)\subset\reali^{n_1}$ of the
action $I^*$, $(\xi,\eta)\in\Vscr(0)$ are Cartesian variables defined
in a neighbourhood $\Vscr(0)\subset \complessi^{2n_2}$ of the origin
and $\epsilon$ is a small parameter. The
Hamiltonian~\eqref{frm:H-modello} is assumed to be analytic in all
variables and in the small parameter $\epsilon$.  The unperturbed
Hamiltonian $H_0$ is assumed to be the sum of a generic Hamiltonian
$h_0(I)$ that can be expanded in power series of $J=I-I^*$ as
\begin{equation}
h_0(J;I^*) = \langle \hat\omega(I^*), J\rangle + \frac{1}{2} \langle
D_I^2h_0(I^*)J,J\rangle + h.o.t.\ ,\qquad\hbox{with}\quad
\hat\omega(I^*)=\parder{h_0}{I}\bigg|_{I=I^*}\ ,
\label{frm:h0expansion}
\end{equation}
and an Hamiltonian $g_0(\xi,\eta)$ that has an elliptic equilibrium at the origin, i.e.,
\begin{equation}
  g_0(\xi,\eta) = \sum_{j=1}^{n_2} \imunit\Omega_j \xi_j \eta_j + h.o.t.\ .
  \label{frm:g0expansion}
\end{equation}

In this paper we investigate the problem of the continuation of
periodic orbits which survive to the breaking of a completely resonant
$n_1$-dimensional torus $I^*$ of~\eqref{frm:H-modello}.  A
  typical example is provided by physical models described by a
  Hamiltonian ~\eqref{frm:H-modello} made by identical and weakly
  coupled nonlinear oscillators (see for example the editorial
  review~\cite{PalP19} on Hamiltonian Lattices), with $n_1$~ones that
  have been excited and oscillate periodically with the same
  frequencies $\hat\omega(I^*)$ and $n_2$~ones are at rest.

The existence of sub-tori surviving to $n_1$-dimensional
  partially resonant tori has been widely treated in the
  literature; see, e.g.,~\cite{Gra74,Tre91,CheW99}, where some
nondegeneracy assumptions on one or both of the Hessian
$D^2_Ih_0(I^*)$ (Kolmogorov nondegeneracy) and of the critical points
of the time-averaged\footnote{Notice that $H_1$ is turned into a
  function of time once evaluated on the periodic flow given by the
  unperturbed resonant torus $I^*$; hence it can be time-averaged over
  the period $T$ of the unperturbed flow.}  perturbation
$\inter{H_1}_T$ (Poincar\'e nondegeneracy) are assumed. Other more
recent works investigate the problem when degeneracy of
$D^2_Ih_0(I^*)$ occurs, like in
\cite{VoyI99,LiYi03,HanLiYi06,HanLiYi10} or in the recent works
\cite{XuLiYi17,XuLiYi18}, where degeneracy is due to different time
scales in the integrable Hamiltonian (like in problems of Celestial
Mechanics).

Differently from the previous literature, our attention is on the
problem of continuation of periodic orbits when Poincar\'e
nondegeneracy does not hold; this typically happens when critical
points of $\inter{H_1}_T$ are not isolated, being part of a
$d$-parameter family. Our perspective is then to look for a normal
form construction which allows us to inspect the Poincar\'e
degeneracy, in the case of completely resonant lower dimensional tori:
such a perturbation approach is able to identify unperturbed periodic
orbits which are candidates for continuation at $\epsilon\neq 0$, as
well as to show the structure of the linear dynamics around those
orbits. In this sense, the results here included represent the natural
generalization of \cite{PenSD18}, where the same problem was faced
limiting to resonant tori of maximal dimension, thus extending the
original ideas of Poincar\'e, see~\cite{PoiI,PoiOU7}. An informal
statement that sums up our results is the following

\begin{quote}  
Consider the Hamiltonian~\eqref{frm:H-modello} with $H_0$ as specified
in~\eqref{frm:h0expansion} and~\eqref{frm:g0expansion}.  Take an
unperturbed resonant torus carrying periodic orbits with frequency
$\omega$, where $\omega$ is such that $\hat\omega(I^*) = \omega k$
with $k\in\interi^{n_1}$.  Assume that the frequency $\omega$ and
the transverse frequency vector $\Omega=\{\Omega_j\}$ are
strongly nonresonant and satisfy the first and second Melnikov
condition.  Assume also that $h_0(I)$ is nondegenerate.  Then, there
exists $\epsilon^*>0$ such that for $|\epsilon|<\epsilon^*$ the
following statements hold true:
\begin{itemize}
\item the Hamiltonian can be put in normal form up to a finite
  arbitrary order $r$ by means of an analytic
  canonical transformation. \emph{Typically} the normal form allows to
  identify isolated approximate periodic orbits that can survive the
  breaking of the unperturbed torus;
\item under suitable assumptions on the spectrum of the approximate
  monodromy matrix, given by the truncated normal form, the
  approximate periodic orbit can be continued for $\epsilon\neq0$;
\item under stricter conditions on the spectrum, the linear
  stability of the true periodic orbit can be inferred from the
  approximate one.
\end{itemize}
\end{quote}

In order to illustrate our original approach, we propose a pedagogical
example with few degrees of freedom, which is inspired to the problem
of the existence of discrete solitons in discrete NLS models.  The
example is described in the following subsection, and is analyzed in
detail in Section~\ref{s:appl}; the corresponding calculations have
been developed with the help of Mathematica Software. Applications of
the normal form to proper dNLS models, with a sufficiently large
number of sites and a suitable variety of degenerate spatially
localized configurations (like vortexes, multi-peaked solutions or
different resonances), would require a systematic investigation with
longer algebraic manipulations; this will be the object of a distinct
and subsequent publication.

\subsection{The \emph{seagull} example}\label{sbs:1.1}

Consider a system of $5$ coupled anharmonic oscillators with Hamiltonian
\begin{equation}
  \label{e.ex.dnls}
  H = H_{0}+\varepsilon H_{1} = \sum_{j=-2}^{2}\Biggl( \frac{x_j^2 +
    y_j^2}{2} + \gamma\biggl( \frac{x_j^2 + y_j^2}{2}\biggr)^2 \Biggr)
  +\epsilon \sum_{j=-2}^{1} (x_{j+1}x_j+ y_{j+1}y_j) \ ,\quad
  (x,y)\in\reali^{5}\times\reali^{5}\ ,
\end{equation}
with $\gamma\neq0$ a parameter tuning the nonlinearity; considering
the figure below, one could consider in an equivalent way a chain of
$7$ masses, i.e., $(x,y)\in\reali^7\times\reali^7$ and with fixed boundary
conditions $x_{-3}=y_{-3}=x_{3}=y_{3}=0$.  We introduce action-angle
variables $ (x_j,y_j) =
(\sqrt{2I_j}\cos\varphi_j,-\sqrt{2I_j}\sin\varphi_j)$, for the set of
indices $j\in\Iscr=\{-2,-1,1,2\}$, and the complex canonical
coordinates
$$
 x_0 = \frac{1}{\sqrt{2}}(\xi_0 + \imunit \eta_0), \quad y_0 =
\frac{\imunit}{\sqrt{2}}(\xi_0 - \imunit \eta_0) \ ,
$$ for the remaining central one $(x_0,y_0)$, so that the Hamiltonian
reads as~\eqref{frm:H-modello} with 
\begin{displaymath}
  \begin{aligned}
    h_0(I) &= \sum_{j\in\Iscr}\left(I_j+\gamma I_j^2
    \right)\ ,\\ g_0(\xi,\eta) &= \imunit\xi_0\eta_0
    -\gamma\xi_0^2\eta_0^2\ ,\\
    H_1 &=2\sqrt{I_{-1} I_{-2}}\cos(\varphi_{-1}-\varphi_{-2})+
    2\sqrt{I_2 I_1}\cos(\varphi_2-\varphi_1)+  \\
    &\quad  +(\xi_0 + \imunit \eta_0) \Bigl(\sqrt{I_{-1}} \cos(\varphi_{-1})+\sqrt{I_1} \cos(\varphi_1)\Bigr)
            -\imunit( \xi_0 - \imunit \eta_0 ) \Bigl(\sqrt{I_{-1}} \sin(\varphi_{-1})+\sqrt{I_1} \sin(\varphi_1)\Bigr)\ .
\end{aligned}
\end{displaymath}
Consider now the 4-dimensional unperturbed resonant torus $I=I^*$ with
${I}_j^*=I_l^*$, for $j,l\in\Iscr$, and $\xi_0=\eta_0=0$. The
configuration is represented in the following picture, which explains the
name seagull

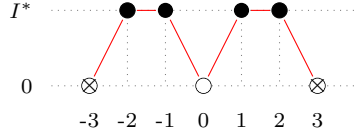
\begin{center}
\begin{tikzpicture}[scale = 0.5]
\begin{scope} 
  \foreach \name/\posX/\posY in {A/0/0,B/1/2,C/2/2,D/3/0,E/4/2,F/5/2,G/6/0}
    \node (\name) at (\posX,\posY) {};
    \draw [dotted, black!50] (-1,0) -- (7,0);
    \draw [dotted, black!50] (-1,2) -- (7,2);
    \draw [dotted, black!50] (1,2) -- (1,0);
    \draw [dotted, black!50] (2,2) -- (2,0);
    \draw [dotted, black!50] (4,2) -- (4,0);
    \draw [dotted, black!50] (5,2) -- (5,0);
\draw (A) circle (2mm) \foreach \n in {B,D,F,G} {(\n) circle (2mm)};
\fill [black] \foreach \n in {B,C,E,F} {(\n) circle (2.0mm)};
\fill [white] \foreach \n in {A,D,G} {(\n) circle (1.9mm)};
  \draw [red] (A) -- (B) -- (C) -- (D) -- (E) -- (F) -- (G);
  \node[label=below:{\footnotesize-3}] at (A) {$\times$};
  \node[label=below:{\footnotesize0}] at (D) {\phantom{$\times$}};
  \node[label=below:{\footnotesize3}] at (G) {$\times$};
  \node[label=below:{\footnotesize-2}] at (1,0) {\phantom{$\times$}};
  \node[label=below:{\footnotesize-1}] at (2,0) {\phantom{$\times$}};
  \node[label=below:{\footnotesize1}] at (4,0) {\phantom{$\times$}};
  \node[label=below:{\footnotesize2}] at (5,0) {\phantom{$\times$}};
  \node[label=left:{\footnotesize $0$}] at (-1,0) {};
  \node[label=left:{\footnotesize $I^*$}] at (-1,2) {};

\end{scope}
\end{tikzpicture}
\end{center}
where the central oscillator is free to move, while the first and last
one are kept at rest due to the Dirichlet boundary conditions. This
case provides a typical and easy mechanism for Poincar\'e degeneracy,
due to the absence of the $1$:$1$ resonance among the nonlinear
oscillators $I_{-1}$ and $I_1$ in the perturbation $\epsilon H_1$ (see
also \cite{PenKSKP19}). Indeed these two oscillators interact at order
$\Oscr(\epsilon)$ only with the central one $(\xi_0,\eta_0)$, which is at
rest in the unperturbed dynamics; as a consequence $\inter{H_1}_T$ is
independent of the phase difference $\varphi_1-\varphi_{-1}$, and its
critical points are not isolated.

In order to reveal a finer structure of the dynamics around the
unperturbed low-dimensional torus, we expand $H$ in power series of
$J=I-I^*$ and introduce the resonant angles $\hat q=(q_1,q)$ and their
conjugate actions $\hat p=(p_1,p)$ as
\begin{equation}
\label{e.4sites.tr}
  \left\lbrace \begin{aligned}
  & q_1 = \varphi_{-2} \\ 
  & q_2 = \varphi_{-1}- \varphi_{-2} \\ 
  & q_3 = \varphi_1 - \varphi_{-1} \\ 
  & q_4 = \varphi_2 -\varphi_1
\end{aligned}
\right. \qquad\qquad \left\lbrace \begin{aligned} 
  & p_1 = J_{-2} + J_{-1} + J_1+ J_2 \\ 
  & p_2 = J_{-1} + J_1 + J_2\\ 
  & p_3= J_1 + J_2 \\ 
  & p_4 = J_2
\end{aligned}
\right. \ .
\end{equation}
Besides using the change of coordinates above, we decide to split the Hamiltonian in the form
\begin{displaymath}
\begin{aligned}
H^{(0)} =\, &\omega p_1 + \imunit\xi_0\eta_0 + f_4^{(0,0)} 
\\
&+f_0^{(0,1)}+ f_1^{(0,1)}+ f_2^{(0,1)}+ f_3^{(0,1)} +
f_4^{(0,1)} + \sum_{\ell>4}  f_\ell^{(0,1)} + \Oscr(\epsilon^2)\ ,
\end{aligned}
\end{displaymath}
where $\omega=1+2\gamma I^*$ is the frequency of any periodic orbit on
the unperturbed torus $p=0$ and $f^{(r,s)}_\ell$ is a polynomial of
degree $l$ in ${\hat{p}}$ and degree $m$ in $(\xi_0,\eta_0)$ with
$\ell=2l+m$ and with coefficients depending on the angles $\hat{q}$.
The index $r\geq 0$ identifies the order of normalization ($r=0$ being
the original Hamiltonian), while $s$ keeps track of the order in the
small parameter $\epsilon$.  The explicit form of $f_0^{(0,1)}$ is
$$
f_0^{(0,1)} =2I^*\epsilon \left(\cos(q_2)+ \cos(q_4)\right)\ .
$$
The splitting of the Hamiltonian in such a form may seem quite
obscure now.  However, considering the equation of motion restricted
to the lower dimensional torus $p=0$, $\xi_0=\eta_0=0$ a moment's
thought suggests how to put in evidence the relevant terms of the
perturbation.
 
In order to continue the periodic orbit surviving the breaking of the
unperturbed lower dimensional torus, the standard approach consists in
averaging the leading term of the perturbation, namely $f_0^{(0,1)}$,
with respect to the fast angle $q_1$ and to look for critical points
of the averaged function on the torus $\toro^3$.  In this specific
example however no averaging is required as $f_0^{(0,1)}$ does not
depend on $q_1$, due to the rotational symmetry typical of dNLS models.
Still, solutions of $\nabla_{q} f_0^{(0,1)}=0$ are not isolated and
appear as $1$-parameter families parameterized by $q_3$, hence
Poincar\'e degeneracy occurs.  Let us remark again that here the
degeneracy is due to the lack of the harmonic
$(\varphi_{-1}-\varphi_1)$ in the perturbation at order $\epsilon$,
that entails the independence of $f_0^{(0,1)}$ by $q_3$.

Our aim is to show that only solutions $(q_2,q_3,q_4)$ with
$q_j\in\graff{0,\pi}$ (the so-called \emph{in} or \emph{out-of} phase
solutions) can be continued for $\epsilon\neq 0$. To this end we
implement a normal form construction that is reminiscent of the
Kolmogorov algorithm (see also~\cite{SanLocGio-2010,GioLocSan-2014}).
Indeed, we perform a sequence of canonical transformations in order to
remove the terms $f_1^{(0,1)}$ and $f_3^{(0,1)}$ and to average the
terms $f_0^{(0,1)}$, $f_2^{(0,1)}$ and $f_4^{(0,1)}$ over the fast
angle $q_1$.  In addition, we perform a translation of the actions
$\hat{p}$ so as to keep fixed the linear frequency $\omega$.

This procedure brings the Hamiltonian in normal form at order
$\epsilon$.  Iterating twice the procedure we get the Hamiltonian in
normal form at order $\epsilon^2$ that reads
\begin{displaymath}
\begin{aligned}
H^{(2)} =\, &\omega p_1 + \imunit\xi_0\eta_0 + f_4^{(2,0)} 
\\
&+ f_0^{(2,1)}+ f_2^{(2,1)}+ f_4^{(2,1)} + \sum_{\ell>4} f_\ell^{(2,1)} \\
&+ f_0^{(2,2)}+  f_2^{(2,2)}+  f_4^{(2,2)} + \sum_{\ell>4}  f_\ell^{(2,2)} + \Oscr(\epsilon^3)\ .
\end{aligned}
\end{displaymath}

Considering the normal form truncated at order two, i.e., neglecting
terms of order $\Oscr(\epsilon^3)$, the leading terms of the
perturbation\footnote{Let us stress that the parameters
    $q_2^*$ and $q_4^*$ allow to select the approximate periodic
    orbit, see equation~\eqref{frm:qstar}.  These parameters are
    introduced by the translation of the actions $\hat{p}$ outlined
    above, see Proposition~\ref{pro:forma-normale-1}
    and~\eqref{frm:traslazioneR} for more details.} are
$$
 f_0^{(2,1)}+ f_0^{(2,2)} = 
2I^*\epsilon\left(\cos(q_2)+ \cos(q_4)\right) +
\frac{\epsilon^2}{\gamma}\bigl(\cos(q_3)-\cos(q_2)\cos(q_2^*)-\cos(q_4)\cos(q_4^*)\bigr)\ ,
$$
and looking for the critical point one gets
\begin{equation}
\label{e.ex.qstar}
2 I^* \sin(q_j)-\frac{\epsilon}{\gamma}\sin(q_j)\cos(q_j^*) =0\ , \quad\hbox{for }j=2,4\ ,\quad\hbox{and}\quad
\epsilon\dfrac{\sin(q_3)}{\gamma} = 0\ .
\end{equation}

The normal form at order two, as we see from the previous equations,
introduces the dependence on $q_3$ that was missing at order one.
With standard arguments of bifurcation theory (the same already used
for example in \cite{PenSPKK18,PenKSKP19}) it is possible to prove
that for $\epsilon$ small enough all families break down and only
solutions with $q_j\in\graff{0,\pi}$ survive; continuation then
follows by means of Newton-Kantorovich fixed point method, since
suitable spectral conditions are verified. We refer the reader to
Section~\ref{s:appl} for details and for interesting results about the
role of $\gamma$ in the linear stability analysis of the solutions:
indeed we will show that neither changing the sign in the nonlinear
parameter $\gamma$ (from focusing to defocusing) nor in the coupling
parameter $\epsilon$ (from attractive to repulsive) affects the nature
of the degenerate eigenspace related to $q_3$, while nondegenerate
directions (as already known from the literature, see
\cite{PelKF05,KouK09}) switch from saddle to center depending on the
sign of the product $\gamma\epsilon$.

Let us remark that the previous example might be explored with a
different approach, which exploits the dNLS structure
of~\eqref{e.ex.dnls}, namely its second conserved quantity
$\sum_{j=-2}^2(x_j^2+y_j^2)$ and the discrete soliton ansatz, which
separate time and (discrete) space variables (see
\cite{JohAGCR98,HenT99,PelKF05,Kev_book09,KouK09,PelKF05b,ChoCMK11}). The
same problem is investigated also in discrete Klein-Gordon models, but
with different perturbation techniques (Lyapunov-Schmidt decomposition
as in \cite{PelS12,PenKSKP19} and Hamiltonian averaging as in
\cite{Aub97,AhnMS01,KouK09,KouKCR13,Kou13}). However, up to our
knowledge, the existing results are valid for specific configurations
(e.g., restricting to consecutive oscillators) and degenerate
solutions can be hardly explored (see \cite{CueKKA11}).  Moreover,
available methods for nondegenerate solutions (as in \cite{Kap01,
  AhnMS01,KouK09}) can be recovered by a single step of our normal
form approach.

The formal statements of the three main results, i.e., the normal form
construction, the continuation theorem and the linear stability
theorem, are detailed in Section~\ref{sec:results} as
Proposition~\ref{pro:forma-normale-1},
Theorem~\ref{teo:forma-normale-r} and Theorem~\ref{t.stab.split},
respectively. Sections~\ref{sec:nfalg} and~\ref{sec:analytics} provide
the description of the normal form construction together with some
analytical estimates. Section~\ref{s:appl} treats extensively the
example~\eqref{e.ex.dnls}. A concluding Appendix collects some
technical results.


\section{Main results}
\label{sec:results}

In this section we introduce the analytic setting and we precisely
state the main results of the paper.

\subsection{Analytic setting}

Consider the distinguished classes of functions
$\widehat{\Pset}_{l,m}$, with integers $l$ and $m$, which can be
written as a Taylor-Fourier expansion
\begin{equation}
g(I,\varphi,\xi,\eta) =
\sum_{i\in\naturali^{n_1} \atop |i|= l}
\,\sum_{(m_1,m_2)\in\naturali^{2n_2} \atop |m_1|+|m_2|=m} 
\, \sum_{{\scriptstyle{k\in\interi^{n_1}}}}
g_{i,m_1,m_2,k}\,I^{i} \exp(\imunit
\langle k, \, \varphi \rangle)\xi^{m_1} \eta^{m_2} \ ,
\label{frm:funz}
\end{equation}
with coefficients $g_{i,m_1,m_2,k}\in\complessi$. 
We say that $g\in\Pset_\ell $ in case
$$ g\in\bigcup_{\substack{l\geq 0, m\geq 0
    \\ 2l+m=\ell}}\widehat{\Pset}_{l,m} \ .
$$ We also set $\Pset_{-4}=\Pset_{-3}=\Pset_{-2}=\Pset_{-1}=\{0\}$;
moreover, we introduce the following notation for those terms which
are independent of both actions $I$ and Cartesian variables
$(\xi,\eta)$
\begin{equation}
\label{frm:hat-tilde-def}
f_{l,0}\in{\widehat\Pscr}_{l,0}\ ,\qquad\qquad f_{0,m}\in{\widehat\Pscr}_{0,m}\ .
\end{equation}

Consider the Hamiltonian~\eqref{frm:H-modello} and select a specific
completely resonant elliptic lower dimensional torus for the
unperturbed Hamiltonian, i.e., set $\xi=\eta=0$ and $I=I^*$ such that
\begin{equation}
\label{e.freq.res}
\hat\omega(I^*)=\parder{h_0}{I}\bigg|_{I=I^*} = \omega k\ ,
\quad\hbox{with }
\omega\in\reali\ ,\ k\in\interi^{n_1}\ .
\end{equation}
Expanding the Hamiltonian in Taylor series of the translated actions
$J=I-I^*$ and the Cartesian coordinates $(\xi,\eta)$, and in Fourier
series of the angles $\varphi$ one has
\begin{equation}
  \label{frm:H(start)}
H^{(0)} = \langle {\hat\omega}, J\rangle
+\sum_{j=1}^{n_2} \imunit\Omega_j\xi_j \eta_j
+\sum_{\ell> 2} f_\ell^{(0,0)}(J,\xi,\eta)  
+\sum_{s>0}\sum_{\ell\geq0}f_\ell^{(0,s)}(J,\varphi,\xi,\eta )\ ,
\end{equation}
where $f_\ell^{(0,s)}\in\Pset_\ell$ and is of order
$\Oscr(\epsilon^s)$. The superscript~0 indicates that the Hamiltonian
is the starting one and in the following will be used to keep track of
the normalization order.

\subsection{The normal form}
\label{ssec:nf}
We define the $(n_1-1)$-dimensional resonant module associated to the
resonant frequency $\hat\omega(I^*)$ in~\eqref{e.freq.res} as
\begin{displaymath}
  \Mscr_{\omega} = \Bigl\{ h\in\interi^{n_1}: \langle\hat\omega(I^*), h\rangle
  =0\Bigr\}\ .
\end{displaymath}
In a neighbourhood of the resonant torus, it is useful to introduce
the resonant variables $(\hat p, \hat q)$ in place of
$(J,\varphi)$. Without affecting the generality of the result, we will
assume $k_1=1$ (see~\eqref{e.freq.res}); this choice simplifies the interpretation of the new
variables.  Given $k\in\interi^{n_1}$ defined by \eqref{e.freq.res},
the canonical change of coordinates is built with an unimodular matrix
which defines the \emph{slow} angles $\hat{q}_{j}= k_j \varphi_1 -
\varphi_j$, for $j=2,\ldots,n_1$, as the phase differences with
respect to the \emph{fast} angle $\hat{q}_1$ of the periodic orbit;
the momenta are defined so as to complement the canonical change of
coordinates, in particular $\hat p_1=\inter{k, J}$.

In order to distinguish the dependence on fast and slow variables in
the normal form construction, we introduce the notations
$\hat{p}=(p_1,p)$, $\hat{q}=(q_1,q)$ with $p_1=\hat{p}_1$,
$p=(\hat{p}_2,\ldots,\hat{p}_n)$ and correspondingly for $q_1$ and
$q$. The Hamiltonian~\eqref{frm:H(start)} then reads
\begin{equation}
  H^{(0)} = \omega p_1 +\sum_{j=1}^{n_2} \imunit\Omega_j\xi_j \eta_j
  +\sum_{\ell> 2} f_{\ell}^{(0,0)}(\hat{p},\xi,\eta) +
  \sum_{s>0}\sum_{{\ell\geq 0}}f_{\ell}^{(0,s)}(\hat{p},\hat{q},\xi,\eta)\ ,
    \label{frm:H(0)}
\end{equation}
where $f_{\ell}^{(0,s)}\in\Pset_\ell$ and it is a function of order
$\Oscr(\epsilon^s)$.  Indeed, the linear change of coordinates keeps
unchanged the classes of functions $\Pset_\ell$.

We introduce the usual complex domains $\Dscr_{\rho,\sigma,R} =
\Gscr_{\rho} \times \toro^{n_1}_{\sigma} \times \Bscr_R$, namely
$$
\begin{aligned}
\Gscr_{\rho}&=\big\{\hat{p}\in\complessi^{n_1}:\max_{1\le j\le
  n_1}|\hat{p}_j|<\rho\big\}\ ,\\ \toro^{n_1}_{\sigma}&=\big\{\hat{q}\in\complessi^{n_1}:\realpart
\hat{q}_j\in\toro,\break\ \max_{1\le j\le n_1}|\imaginary
\hat{q}_j|<\sigma\big\}\ ,
\\ \Bscr_{R}&=\big\{(\xi,\eta)\in\complessi^{2n_2}:\max_{1\le j\le
  n_2} \tond{|\xi_j|+|\eta_j|}< R \big\}\ .
\end{aligned}
$$

Given a generic analytic function
$g:\Dscr_{\rho,\sigma,R}\to\complessi$, we define the
weighted Fourier norm
\begin{equation*}
\|g\|_{\rho,\sigma,R}=
\sum_{i\in\naturali^{n_1}}
\,\sum_{(m_1,m_2)\in\naturali^{2n_2}}
\, \sum_{{\scriptstyle{k\in\interi^{n_1}}}}
|g_{i,m_1,m_2,k}| \rho^{|i|} R^{|m_1|+|m_2|} e^{|k|\sigma}\ ;
\end{equation*}
hereafter, we are going to use the shorthand notation
$\|\cdot\|_{\alpha}\,$ for $\|\cdot\|_{\alpha(\rho,\sigma)}\,$.

We now state our main result on the normal form construction; the
proof is deferred to Section~\ref{sec:analytics}.

\begin{proposition}\label{pro:forma-normale-1}
Consider the Hamiltonian $H^{(0)}$, expanded as in~\eqref{frm:H(0)},
and being analytic in the domain $\Dscr_{\rho,\sigma,R}$.  Assume that
\begin{description}
\item[H1)]
  there exists a positive constant $m$ such that for every
  $v\in\reali^{n_1}$ one has
  \begin{equation}
    m \sum_{i=1}^{n_1} |v_i| \leq \sum_{i=1}^{n_1} \bigl|\sum_{j=1}^{n_1} C_{0,ij}
    v_j\bigr|\ ,\quad\hbox{where}\quad C_{0,ij}=\frac{\partial^2
      f_4^{(0,0)}}{\partial \hat{p}_i \partial \hat{p}_j}\bigg|_{\hat p=0}\ ;
    \label{frm:twist}
  \end{equation}
\item[H2)]  the terms appearing in the expansion of the Hamiltonian satisfy
  \begin{equation}
  \|f_\ell^{(0,s)}\|_1 \leq \frac{E}{2^\ell}
  \epsilon^s\ ,\qquad\hbox{with}\quad E>0\ ,\ l,s\geq 0\ .
  \label{frm:decadimento-iniziale}
  \end{equation}
\item[H3)] the frequencies $\omega$ defined in \eqref{e.freq.res} and
  $\Omega_l$ introduced in \eqref{frm:g0expansion} satisfy the first
  and second nonresonance Melnikov conditions
   \begin{align} 
     k_1\omega\pm \Omega_j &\neq 0, \qquad k_1\in\interi \ ,
    \label{melnikov1}
    \\ k_1\omega\pm \Omega_l \pm \Omega_k &\neq 0, \qquad
    k_1\in\interi\setminus \lbrace 0\rbrace \ .
    \label{melnikov2}
   \end{align}
\end{description}
Then, for every integer $r\geq 1$ there exists
$\epsilon_{r}^*>0$ such that for $|\epsilon| < \epsilon_{r}^*$
there exists an analytic canonical transformation $\Phi^{(r)}$
satisfying
\begin{equation}
\Dscr_{\frac{1}{4}(\rho,\sigma,R)} \subset
\Phi^{(r)}\Bigl(\Dscr_{\frac{1}{2}(\rho,\sigma,R)}\Bigr) \subset
\Dscr_{\frac{3}{4}(\rho,\sigma,R)}
\label{frm:domini}
\end{equation}
such that the Hamiltonian $H^{(r)}= H^{(0)} \circ \Phi^{(r)}$ has the
following expansion in normal form up to order $r$
\begin{equation}
  \begin{aligned}
    H^{(r)}(\hat{q},\hat{p},\xi,\eta;q^*) =& \omega p_1 + \sum_{j=1}^{n_2} \imunit \Omega_j\xi_j\eta_j
    +\sum_{\ell> 2} f_{\ell}^{(r,0)}(\hat{p},\xi,\eta)\cr
    &+\sum_{s=1}^{r} \Bigl(f_{0}^{(r,s)}(q;q^*) +
    f_{2}^{(r,s)}(q,\hat{p},\xi,\eta;q^*) + f_{3}^{(r,s)}(\hat{q},\xi,\eta;q^*) +
    f_{4}^{(r,s)}(\hat{q},\hat{p},\xi,\eta;q^*)\Bigr) \cr
    &+\sum_{s>r} \sum_{\ell=0}^{4} f_{l}^{(r,s)}(\hat{q},\hat{p},\xi,\eta;q^*)+\sum_{s>0}\sum_{\ell>4}f_{l}^{(r,s)}(\hat{q},\hat{p},\xi,\eta;q^*)\ ,
    \label{frm:H(r)}
  \end{aligned}
\end{equation}
where $q^*\in\toro^{n_1-1}$ is a fixed but arbitrary vector of
parameters. The Hamiltonian~\eqref{frm:H(r)} is said to be in normal form up to order $r$ since for $s\leq r$ it satisfies

\begin{enumerate}
\item $f_{0}^{(r,s)}(q;q^*)$ do not depend on the fast angle $q_1$;

\item $f_{1}^{(r,s)}(\hat{q},\xi,\eta;q^*)$ do not appear;

\item $f_{2}^{(r,s)}(q,\hat{p},\xi,\eta;q^*)$ do not
  depend on $q_1$ and, evaluated at $\xi=\eta=0$
  and $q=q^*$, are equal to zero;

\item $f_{3}^{(r,s)}(\hat{q},\hat p, \xi,\eta;q^*)$ do not
  depend on the actions ${\hat p}$;

\item $f_{4}^{(r,s)}(\hat{q},\hat{p},\xi,\eta;q^*)$, evaluated at
  $\xi=\eta=0$, do not depend on the fast angle $q_1$.
\end{enumerate}  

\end{proposition}

\noindent
Some comments are in order. Assumption {\bf H1} is needed in order to
keep the frequency fixed on the torus, as in the classical Kolmogorov
construction.  It implies the invertibility of the Hessian $C_0$,
which is equivalent to the invertibility of $D_J^2{h_0}(I^*)$ in the
original coordinates: this is sometimes known as \emph{twist}
condition, or \emph{Kolmogorov nondegeneracy} condition, and encodes
the fact that the resonant torus is locally isolated in the space of
actions. Assumption~{\bf H2} is a typical requirement on the decay of
the homogeneous terms of the Taylor expansion in $\epsilon$. The last
ones, {\bf H3}, are the so-called first and second Melnikov
conditions, and ensure absence of resonances between the periodic
motion and the transverse linear oscillators. Actually, the first
Melnikov condition~\eqref{melnikov1} is enough to get existence of the
continuation of the periodic orbit, while the second
one~\eqref{melnikov2} is needed to exhibit the linear stability of the
orbit (see \cite{Giorgilli-2012}).

The proof of Proposition~\ref{pro:forma-normale-1} is based on
standard arguments in Lie series theory.  The key estimates that
allow to complete the proof are reported in Lemma~\ref{lem:lemmone.1}
in Section~\ref{sec:analytics}. We do not report here all the
(tedious) details since similar results have been already published
in,
e.g.,~\cite{GioLoc-1997,Gio03,Giorgilli-2012,GioLocSan-2014,GioLocSan-2015,SanCec-2017,PenSD18}.

\subsection{Approximation and continuation of periodic orbits}
\label{ssec:cpo}

The Hamiltonian~\eqref{frm:H(r)}, being in normal form up to order
$r$, allows to find approximate periodic orbits.  Precisely, consider
the normal form approximation $Z^{(r)}$, i.e., $H^{(r)}$ neglecting
the terms of order $\Oscr\tond{\epsilon^{r+1}}$,
\begin{equation}
  \begin{aligned}
    Z^{(r)}(\hat{q},\hat{p},\xi,\eta;q^*) =\,& \omega p_1 + \sum_{j=1}^{n_2} \imunit \Omega_j\xi_j\eta_j
    +\sum_{\ell> 2} f_{\ell}^{(r,0)}(\hat{p},\xi,\eta)\cr
    &+\sum_{s=1}^{r} \Bigl(f_{0}^{(r,s)}(q;q^*) +
    f_{2}^{(r,s)}(q,\hat{p},\xi,\eta;q^*) + f_{3}^{(r,s)}(\hat{q},\xi,\eta;q^*) +
    f_{4}^{(r,s)}(\hat{q},\hat{p},\xi,\eta;q^*)\Bigr) \cr
    &+\sum_{s=0}^r\sum_{\ell>4}f_{l}^{(r,s)}(\hat{q},\hat{p},\xi,\eta;q^*)\ ,
    \label{frm:Z(r)}
  \end{aligned}
\end{equation}
and take as initial datum
$x^*=(q=q^*,\hat p= 0,\xi=0,\eta=0)$.  It is straightforward to see that the canonical equations read
\begin{equation*}
\dot{q}_1 = \omega\ ,  \qquad
\dot{q} = 0\ , \qquad
\dot{p}_1 = 0\ ,  \qquad
\dot{p} = -\sum_{s=1}^r \nabla_{q} f_0^{(r,s)}\Big|_{q=q^*}\ ,  \qquad
\dot{\xi}=0 \ , \qquad
\dot{\eta}=0
\ .
\end{equation*}
Hence, if $q=q^*$ is chosen as a solution of
\begin{equation}
\label{frm:qstar}
\sum_{s=1}^r \nabla_q {f_0^{(r,s)}}\big|_{q=q^*} = 0\ ,  
\end{equation}
then $(q_1(0),x^*)$ is the initial datum (modulo the initial phase
$q_1(0)$) of a periodic orbit with frequency $\omega$ for the
truncated normal form, $x=x^*$ being a relative
equilibrium\footnote{Let us stress that, for $r>1$, $q^*(\epsilon)$
  actually depends on $\epsilon$ and it is analytic.  Indeed, $(\omega
  t+q_1(0),x^*(\epsilon))$ is a periodic solution of an analytic
  Hamiltonian whose flow is analytic
  in $\epsilon$.} for $Z^{(r)}$.

We now introduce the smooth map
$\Upsilon(x):\Uscr(x^*)\subset\reali^{2n-1}\to
\Vscr(x^*)\subset\reali^{2n-1}$ as
\begin{equation}\label{frm:Ups}
\Upsilon(x(0);\epsilon,q_{1}(0))=
  \begin{pmatrix}
  q_1(T)-q_1(0)-\omega T\\
  q(T)-q(0)\\
  p(T)-p(0)\\
  \xi(T)-\xi(0) \\
  \eta(T)-\eta(0)
  \end{pmatrix} \ ,
\end{equation}
parameterized by the initial phase $q_1(0)$ and $\epsilon$, with $T$ the period of the periodic orbit.  Then, the
periodicity condition is equivalent to $\Upsilon(x(0);\epsilon,q_{1}(0))=0$; notice that in~\eqref{frm:Ups} we have neglected the equation for $p_1$, due to the
conservation of the energy, which provides a dependence relation among
the $2n$ equations (see for example \cite{MelS05}).

The periodic orbit $x^*$ of the truncated normal form $Z^{(r)}$ turns
out to be an approximate periodic orbit of the full Hamiltonian
system; indeed it will be shown (see Lemma~\ref{lem:stima-approx-p.o.}
in Section~\ref{sec:analytics}) that
\begin{displaymath}
\norm{\Upsilon(x^*;\epsilon,q_{1}(0))}\leq c_1\epsilon^{r+1}\ ,
\end{displaymath}
for some positive constant $c_1(r)$ (growing with $r$) and $\epsilon$
small enough.  A true periodic orbit, close to the approximate one, is
then identified by an initial datum $x^*_{\rm p.o.}\in\Uscr(x^*)$ such
that
\begin{equation*}
  \Upsilon(x^*_{\rm p.o.};\epsilon,q_1(0))= 0\ .
\end{equation*}
In order to prove the existence of a true periodic orbit $x^*_{p.o.}$
close to $x^*$ we apply the Newton-Kantorovich
method (see \cite{KolF89,PenSD18}).

\begin{proposition}[Newton-Kantorovich]
\label{prop:N-K}
Consider a smooth map
$\Upsilon\in\mathcal{C}^1\left(\mathcal{U}(x^*)\times\mathcal{U}(0),V\right)$.
Assume that there exist three positive constants $c_{1}$, $c_{2}$ and $c_{3}$ dependent, for
$\epsilon$ small enough, on $\Uscr(x^*)\subset V$ only, and two
parameters $0\leq 2\alpha<\beta$ such that
\begin{align}
  \|\Upsilon(x^*,\epsilon)\| &\leq c_1 |\epsilon|^{\beta}\ ,\label{equilb-approx}\\
  \opnorm{[\Upsilon'(x^*,\epsilon)]^{-1}} &\leq c_2 |\epsilon|^{-\alpha}\ ,\label{inversa}\\
  \opnorm{\Upsilon'(z,\epsilon)-\Upsilon'(x^*,\epsilon)} &\leq c_3 \norm{z-x^*}\ ,\label{lipsc}
\end{align}
where $\opnorm{\cdot}$ denotes the operator norm.  Then there exist
positive $c_0$ and $\epsilon^*$ such that, for $|\epsilon|<\epsilon^*$, there
exists a unique $x^*_{p.o.}(\epsilon)\in\mathcal{U}(x^*)$ which fulfills
\begin{displaymath}
\Upsilon(x^*_{p.o.},\epsilon)=0\ ,\qquad \|x^*_{p.o.}-x^*\| \leq
c_0|\epsilon|^{\beta-\alpha}\ .
\end{displaymath}
Furthermore, Newton's algorithm converges to $x^*_{p.o.}$.
\end{proposition}

We stress that the main assumption concerns the invertibility of
differential of $\Upsilon$,
\begin{equation}
  \label{frm:M(epsilon)}
 M(\epsilon) = \Upsilon'(x^*;\epsilon, q_1(0))\ ,
\end{equation}
being essentially a condition on the minimum eigenvalue, that is
vanishing with $\epsilon$.  Indeed, it is extremely difficult to directly
verify~\eqref{inversa} on an actual application starting from the
definition~\eqref{frm:M(epsilon)}.

A way out is given by the so-called variational equations, around a
given orbit $\phi^t(x_0)$, where $x_0=\phi^0(x_0)$ is the initial
datum.  It turns out that $M(\epsilon)$ can be derived by the monodromy matrix
$\Phi(T;H^{(r)},x^*)$, which is the evolution at time $T$ of the
fundamental matrix of the linear vector field
$DX_{H^{(r)}}(\phi^t(x^*))$
\begin{displaymath}
\frac{d}{dt} \Phi(t;H^{(r)},x^*) =
DX_{H^{(r)}}(\phi^t(x^*))\Phi(t;H^{(r)},x^*)\ ,
\end{displaymath}
where we have denoted by $X_H$ the Hamiltonian vector field given by
$H$. Actually, one can take advantage of the normal normal form
construction in order to approximate $M(\epsilon)$; indeed, by
considering the truncated normal form $Z^{(r)}$ in~\eqref{frm:Z(r)},
it turns out that the linearization $DX_{Z^{(r)}}(x^*)$ around the
relative equilibrium $x^*$ is a constant matrix and, furthermore, it
is block diagonal. This is a consequence of properties \emph{4}
and~\emph{5} in the normal form construction, which allows to split
the dependence on the ``internal'' variables $(q,\hat p)$ and
$(\xi,\eta)$ in the quadratic part of $Z^{(r)}$. In order to better
develop this point, and show how to exploit $DX_{Z^{(r)}}(x^*)$ to
verify~\eqref{inversa}, we need to introduce a convenient notation.
Let $M$ be a $2n$-dimensional square matrix. We denote by $\matred{M}$
the reduced matrix: the $2(n-1)$ dimensional square matrix obtained
from $M$ by removing the first column (related to the fast angle
$q_1$) and the ($n_1+1$)-th row (related to the momentum $p_1$).

We can now state the following
\begin{lemma}
\label{l.mat.red}
The differential $M(\epsilon)$ defined in~\eqref{frm:M(epsilon)} is
the reduction of $\Phi(T;H^{(r)},x^*)-\Id$, namely
\begin{displaymath}
M(\epsilon) = \matred{\tond{\Phi(T;H^{(r)},x^*)-\Id}}\ .
\end{displaymath}
Moreover, $M(\epsilon)$ has the following decomposition
\begin{equation}
M(\epsilon)=N(\epsilon)+\Oscr(\epsilon^{r+1})\ ,
\qquad\hbox{with}\quad
N(\epsilon)=\left( \begin{matrix}
{N_{11}}(\epsilon) & O\\
O & N_{22}(\epsilon) \\
\end{matrix}\right)\ ,
\label{frm:N11N22}
\end{equation}
where the leading term reads
$$
N(\epsilon) =
\matred{\tond{\Phi(T;Z^{(r)},x^*)-\Id}}\ ,\qquad\hbox{with}\quad
\Phi(T;Z^{(r)},x^*)=\exp\tond{DX_{Z^{(r)}}(x^*) T}\ .
$$
\end{lemma}

\begin{proof}
It is well known that the fundamental matrix $\Phi(T;H^{(r)},x^*)$
equals the differential of the Hamiltonian flow with respect to the
generic initial datum (close to $\phi^t(x^*)$). Since $\Upsilon$
ignores the evolution of $p_1$ and does not depend on the fast angle
$q_1$, we obtain $M(\epsilon) =
\matred{\tond{\Phi(T;H^{(r)},x^*)-\Id}}$. Thus, the structure of
$\Phi(T;H^{(r)},x^*)-\Id$ can be investigated exploiting the
linearization around the relative equilibrium $x^*$ of the truncated
normal form $Z^{(r)}$.  Indeed, the matrix $\Phi(T;H^{(r)},x^*)$ can
be approximated by $\Phi(T;Z^{(r)},x^*)$, the last having a block
diagonal structure and being represented by the exponential of the
time-independent matrix $DX_{Z^{(r)}}(x^*)$.
\end{proof}

In view of Lemma~\ref{l.mat.red}, we can focus on the leading term
$N(\epsilon)$, which is expected to be explicitly calculated via the
normal form algorithm. Moreover, the nonresonance
condition~\eqref{melnikov1} ensures that the matrix $N_{22}(0)$ is
diagonal with eigenvalues of order $\Oscr(1)$.  Thus, by continuity of
the spectrum with respect to $\epsilon$, the same order of magnitude holds
also for the eigenvalues of $N_{22}(\epsilon)$. Accordingly, only the
eigenvalues in $\Sigma(N_{11}(\epsilon))$ vanish with $\epsilon\to 0$, and the
continuation result can be formulated by assuming suitable conditions
on $N_{11}(\epsilon)$.

\begin{theorem}
\label{teo:forma-normale-r}
  Consider the map $\Upsilon$ defined in~\eqref{frm:Ups} in a
  neighbourhood of the lower dimensional torus $\hat{p}=0$, $\xi=\eta=0$ and let
  $x^*(\epsilon)=({q^*}(\epsilon),0,0,0)$, with ${q^*}(\epsilon)$
  satisfying~\eqref{frm:qstar}.  Assume that
  \begin{equation}
    \label{e.small.Ups}
    \norm{\Upsilon(x^*(\epsilon);\epsilon,q_1(0))}\leq c_1
    \epsilon^{r+1}\ ,
  \end{equation}
  where $c_1$ is a positive constant depending on $\Uscr$ and
  $r$. Assume that $N_{11}(\epsilon)$ in~\eqref{frm:N11N22} is invertible
  and there exists $\alpha>0$ with $2\alpha<r+1$ such that
  \begin{equation}
    \label{e.small.eig}
    |\lambda|\gtrsim |\epsilon|^\alpha\ ,\qquad\hbox{for all}\quad\lambda\in
    \Sigma\tond{{N_{11}}(\epsilon)}\ .
  \end{equation}
  Then, there exist $c_0>0$ and $\epsilon^*>0$ such that for any
  $0\leq|\epsilon|<\epsilon^*$ there exists a unique $x^*_{\rm
    p.o.}(\epsilon)=(q^*_{\rm p.o.}(\epsilon),\hat{p}_{\rm
    p.o.}(\epsilon),\xi_{\rm p.o.}(\epsilon),\eta_{\rm
    p.o.}(\epsilon))\in\Uscr$ which solves
  \begin{equation}
    \label{e.exist.appr}
    \Upsilon(x^*_{\rm p.o.};\epsilon,q_1(0))=0
    \ ,\qquad\hbox{with}\quad
    \norm{x^*_{\rm
        p.o.}-x^*}\leq c_0\epsilon^{r+1-\alpha}\ .
\end{equation}
\end{theorem}

\begin{proof}
The proof consists in the application of the Newton-Kantorovich
method, as stated in the Proposition~\ref{prop:N-K}.  Indeed,
condition~\eqref{equilb-approx} holds true with $\beta=r$, because of
Lemma~\ref{lem:stima-approx-p.o.}.  Moreover, condition~\eqref{lipsc}
is also satisfied, in view of the analyticity of the Hamiltonian and
its vector field.  The third and last hypothesis~\eqref{inversa} is
about the invertibility of the Jacobian matrix $M(\epsilon)$ and on
the smallness of its eigenvalues. Due to Lemma~\ref{l.mat.red},
invertibility of $N(\epsilon)$ requires invertibility of the two
blocks $N_{11}(\epsilon)$ and $N_{22}(\epsilon)$; but
$N_{22}(\epsilon)$ is invertible because of the first Melnikov
condition, hence assuming $N_{11}(\epsilon)$ invertible ensures
invertibility of $N(\epsilon)$. Then, thanks to $|\lambda|\gtrsim
|\epsilon|^\alpha$ (with $\alpha<r$, which is guaranteed by the hypothesis
$2\alpha<r+1$) and by exploiting Proposition~\ref{p.min.eigen} in the
Appendix, invertibility is preserved under perturbations of order
$\Oscr(\epsilon^{r+1})$ and $|\nu|\gtrsim |\epsilon|^\alpha$ is ensured for
any $\nu\in\Sigma(M(\epsilon))$, thus proving the validity of~\eqref{inversa}.
\end{proof}

\begin{remark}
Let us note that in the nondegenerate case ($r=1$) it can be shown
(e.g. via a $\sqrt\epsilon$-scaling of the momenta $\hat p$, as in
\cite{Tre91}) that the eigenvalues are of order $\Oscr(\sqrt\epsilon)$ and
the existence of periodic orbits follows by direct application of the
implicit function theorem, taking $\alpha=\frac12$.
\end{remark}

\subsection{Approximate and effective linear stability}
\label{ssec:ls}

We come now to the investigation of the linear stability of the
approximate periodic orbit $x^*$ via the normal form construction. In
plain words, the normal form procedure has removed the time dependence
in the $\Phi(t;H^{(r)},x^*)$ up to order $\Oscr(\epsilon^{r})$, thus
reducing the problem of the approximate stability to the computation
of the eigenvalues of a constant matrix, in place of the Floquet
exponents.  To clarify this point, we introduce the variables $(\hat
Q,\hat P)$, representing the small displacements around the relative
equilibrium
\begin{equation*}
  Q_1 =  q_1 - \omega t - q_1(0)\ ,\qquad
  Q = q-q^*\ , \qquad
  P_1 = p_1\ , \qquad
  P = p\ .
\end{equation*}

Replacing the original variables $(\hat q, \hat p)$ with the new ones
$(\hat Q,\hat P)$ in~\eqref{frm:Z(r)}, i.e., the truncated Hamiltonian
in normal form at order $r$, and keeping only the quadratic terms, one
immediately obtains the Hamiltonian vector field linearized around
$x^*$.  Let us stress that, by normal form construction, the quadratic
Hamiltonian, and hence the linearized equations, is independent of
$Q_1$.

To represent the linearized Hamiltonian vector field, it is convenient
to introduce the matrices
\begin{alignat*}{2}
  B(\epsilon) &= D^2_qZ^{(r)}(\omega t+q_1(0),x^*) \ ,\qquad &G(\epsilon) &= D^2_{\xi}Z^{(r)}(\omega t+q_1(0),x^*)\ ,\\
  D(\epsilon) &= D^2_{q\hat p}Z^{(r)}(\omega t+q_1(0),x^*) \ ,\qquad &F(\epsilon) &= D^2_{\eta}Z^{(r)}(\omega t+q_1(0),x^*)\ ,\\
  C(\epsilon) &= D^2_{\hat p}Z^{(r)}(\omega t+q_1(0),x^*) \ ,\qquad &E(\epsilon) &= D^2_{\xi\eta}Z^{(r)}(\omega t+q_1(0),x^*)\ .
\end{alignat*}
We notice that the above matrices admit the asymptotic (analytic)
expansions in $\epsilon$
\begin{alignat*}{6}
  &B(\epsilon) &&= &&\epsilon B_1 + \Oscr(\epsilon^2)\ ,\qquad      &&G(\epsilon) &&= &&\epsilon G_1 + \Oscr(\epsilon^2)\ ,\\
  &D(\epsilon) &&= &&\epsilon D_1 + \Oscr(\epsilon^2)\ ,\qquad      &&F(\epsilon) &&= &&\epsilon F_1 + \Oscr(\epsilon^2)\ ,\\
  &C(\epsilon) &&=C_0 + &&\epsilon C_1 + \Oscr(\epsilon^2)\ ,\qquad &&E(\epsilon) &&= E_0 + &&\epsilon C_1 +
  \Oscr(\epsilon^2)\ ,\quad\hbox{with}\quad E_0=\text{diag}(\imunit\Omega_j)\ .
\end{alignat*}
In order to formally include the $Q_1$ dependence in the quadratic
Hamiltonian, we extend the matrices $B(\epsilon)$ and $D(\epsilon)$
(the last one being $(n_1-1)\times n_1$ rectangular) to $n_1\times
n_1$ square matrices. Precisely, we denote by $\matext{B(\epsilon)}$
the square-matrix obtained by adding a zero row and column,
respectively at the top and left of $B(\epsilon)$. Similarly, we
denote by $\matext{D(\epsilon)}$ the square-matrix obtained adding a
zero row at top of $D(\epsilon)$.  In this way, the quadratic
Hamiltonian that gives the linear approximation of the dynamics close
to the approximate periodic orbit reads
\begin{equation}
  \label{frm:square-K}
  Z^{(r)}_2 = \frac{1}{2} \langle \matext{B(\epsilon)} \hat Q,  \hat Q \rangle+
  \langle \matext{D(\epsilon)}\hat Q, \hat P\rangle
  +\frac12 \langle C(\epsilon) \hat P,  \hat P\rangle + \frac12 \langle G(\epsilon) \xi,\eta\rangle +
  \langle E(\epsilon) \xi,\eta \rangle + \frac12 \langle F(\epsilon)\eta , \eta \rangle  \ ,
\end{equation}
and the canonical linear vector field can be represented by a block
diagonal matrix $L(\epsilon)$ as
\begin{equation}
\label{frm:L}
L(\epsilon)=
\begin{pmatrix}
L_{11}(\epsilon) & 0\\
0 & L_{22}(\epsilon)
\end{pmatrix}\ ,
\end{equation}
with
$$
L_{11}(\epsilon)=
\begin{pmatrix}
    \phantom{-}\matext{D(\epsilon)}^\top & C(\epsilon)\\
    -\matext{B(\epsilon)} & -\matext{D(\epsilon)}
\end{pmatrix}
\quad\hbox{and}\quad
L_{22}(\epsilon)=
\begin{pmatrix}
    \phantom{-}E(\epsilon)^\top & \phantom{-}F(\epsilon)\\
    -G(\epsilon) & -E(\epsilon)
  \end{pmatrix}\ .
$$

As $L(\epsilon)$ is independent of time, the stability of the
approximate periodic orbit reduces to the study of the spectrum of
$L(\epsilon)$, which splits into two distinct components. The first
one is $\Sigma(L_{11}(\epsilon))$, made of $2n_1$ eigenvalues which
vanish as $\epsilon\to 0$, in view of the $n_1-1$ resonances on the
unperturbed invariant torus $I=I^*$. The second one is
$\Sigma(L_{22}(\epsilon))$.  From now on, we assume that
$L_{22}(\epsilon)$ is positive (negative) definite, so that
$\Sigma(L_{22}(\epsilon))$ is generically made of $n_2$ pairs of
conjugate imaginary eigenvalues $\imunit\Omega_j(\epsilon)$, at least
for $\epsilon$ small enough. The linear stability in the directions
transverse to the lower dimensional torus are guaranteed by the
assumptions on $L_{22}(\epsilon)$ to be positive definite, i.e.,
$\Omega_j(0)>0$ for $j=1,\ldots,n_2$.  Indeed, the origin is a
nondegenerate elliptic equilibrium for the unperturbed Hamiltonian
$g_0$, therefore it persists for $\epsilon$ small enough.  As a
consequence, the approximate linear stability of the periodic orbit
depends only on the \emph{vanishing part} of the spectrum
$\Sigma(L_{11}(\epsilon))$. On the contrary, if some harmonic
oscillators $\imunit\Omega_j\xi_j\eta_j$ are free to rotate with equal
frequencies, but in opposite directions, then the transverse
instability of the approximate periodic orbit can be produced by
collisions of eigenvalues having different Krein signature.

We now investigate to what extent the stability of the \emph{true}
periodic orbit can be inferred by the stability of the approximate
one, under suitable assumptions on $\Sigma(L_{11}(\epsilon))$.  In fact,
the distance between the monodromy matrix
$\Phi(T;H^{(r)},x^*_{{\rm p.o.}})$ and the matrix $\exp(L(\epsilon)T)$
is not only due to the normal form remainder, that is of order
$\Oscr(\epsilon^{r+1})$, but also to the approximation of the periodic
orbit, which is actually dominant, being of order
$\Oscr(\epsilon^{r+1-\alpha})$, with $2\alpha<r+1$. This is part of the
statement claimed in the next Theorem~\ref{t.stab.split}; this
statement also claims that the spectrum
$\Sigma(\Phi(T;H^{(r)},x^*_{{\rm p.o.}}))$ of the monodromy matrix splits into two different components, which are deformations of the
approximate ones.

\begin{theorem}
  \label{t.stab.split}
Consider the monodromy matrix $\Phi(T;H^{(r)},x^*_{{\rm p.o.}})$ and
its approximation given by $\exp\tond{L(\epsilon)T}$, with $L_{22}(\epsilon)$ positive
definite. Then for $|\epsilon|$ small enough the following holds true:
\begin{enumerate}

  \item there exists a positive constant $c_A$ such that one has
    \begin{equation}
      \label{e.mon.appr}
      \Phi(T;H^{(r)},x^*_{{\rm p.o.}}) = \exp\tond{L(\epsilon)T} +
      A\ ,\qquad\hbox{with}\quad \opnorm{A}\leq c_A|\epsilon|^{r+1-\alpha}\ ,
    \end{equation}
    where $\alpha$ is the same as in
    Theorem~\ref{teo:forma-normale-r};
    
  \item $\Sigma(\Phi(T;H^{(r)},x^*_{{\rm p.o.}}))=\Sigma_{11}
    \cup \Sigma_{22}$, where $\Sigma_{11}$ is close to
    $\Sigma(\exp(L_{11}(\epsilon) T))$ and includes at least two elements equal
    to $1$, while $\Sigma_{22}$ is close to $\Sigma(\exp(L_{22}(\epsilon) T))$
    and all its elements lie on the unit circle.
\end{enumerate}
\end{theorem}

\begin{proof}
First observe that, in view of the continuity and the separation of
the two spectra $\Sigma(L_{11}(\epsilon))$ and
$\Sigma(L_{22}(\epsilon))$, the spectrum of the monodromy matrix can
be split into two different parts. Moreover, in view of the
Melnikov nonresonance conditions and the fact that transverse linear
oscillators have frequencies with the same sign, Krein signature (see for
example \cite{YakS75,MacS98,AhnMS01,Meiss17,KenMey09}) ensures that
$\Sigma_{22}$, which is a deformation of $\Sigma(\exp(L_{22}(\epsilon)
T))$, lies on the unit circle.

In order to obtain the estimate of the error in~\eqref{e.mon.appr}, we
exploit the fact that the monodromy matrix is the differential of the
flow with respect to the initial datum.  Considering the matrix
$\Phi(T;Z^{(r)},x^*_{{\rm p.o.}})$, we take into account two different sources of
approximation: the one of the Hamiltonian $H^{(r)}$ with its normal
form $Z^{(r)}$ and the one due to the approximation of the initial
datum of the periodic orbit.  Hence, the error term consists of the
normal form remainder $\Oscr(\epsilon^{r+1})$ and of the error of the
periodic orbit, which is of order $\Oscr(\epsilon^{r+1-\alpha})$ (with
$ 2\alpha<r+1$), as it follows from Theorem~\ref{teo:forma-normale-r}.
The latter is the dominant one and this allows to conclude the proof.

\end{proof}

We remark that in view of Lemma~\ref{l.mat.red}, in
Theorem~\ref{teo:forma-normale-r} we can rewrite $M(\epsilon)$ as
\begin{equation}
  \label{e.M.appr}
M(\epsilon) = \tond{\exp\tond{L(\epsilon)T}-\Id}_{\flat} +
\Oscr(\epsilon^{r+1})\ .  
\end{equation}
As a consequence, the matrix $N_{11}(\epsilon)$ in~\eqref{e.small.eig} reads
$N_{11}(\epsilon) = (\exp(L_{11}(\epsilon)T)-\Id)_{\flat}$.

\medskip
We are now ready to state the result on the localization of the
eigenvalues for $\Sigma_{11}$ (which are all close to~$1$) by
exploiting the spectrum of $L_{11}(\epsilon)$ in the generic case of distinct
eigenvalues. The result is the following (see also Lemma~2
in~\cite{AhnMS01})
\begin{theorem}
\label{t.lin.stab.2}
Assume that $L_{11}(\epsilon)$ has $2n_1-2$ distinct non-zero eigenvalues
and let $\tilde c>0$ and $\beta<r+1-\alpha$, with $2\alpha<r+1$ as in
Theorem~\ref{teo:forma-normale-r}, be such that
\begin{equation}
\label{e.dist.eig}
|\lambda_j-\lambda_k|>\tilde c\epsilon^{\beta}\ ,\qquad\hbox{for
  all}\quad\lambda_{j},\lambda_{k}\in\Sigma(L_{11}(\epsilon))\setminus\graff{0}\ .
\end{equation}
Then there exists $\epsilon^*>0$ such that if $|\epsilon|<\epsilon^*$ and
$\mu=e^{\lambda T}\in\Sigma(\exp(L_{11}(\epsilon) T))$, there exists
one eigenvalue $\nu\in\Sigma_{11}$ inside the complex disk
$D_\epsilon(\lambda)=\graff{z\in\complessi\ :\ |z-\lambda|<c
  \epsilon^{r-\alpha+1}}$, with $c>0$ a suitable constant independent of
$\mu$.
\end{theorem}

\begin{proof}
The proof follows from Proposition~\ref{p.def.eigen} in the Appendix,
by exploiting~\eqref{e.mon.appr} and the fact that the difference
between the Floquet multipliers close to $1$, $e^{\lambda_j
  T}-e^{\lambda_k T}$, is, at leading order, the same as the exponents
$\lambda_j-\lambda_k$.
\end{proof}

\begin{corollary}
\label{c.stab}
  Under the assumptions of Theorem~\ref{t.lin.stab.2} the periodic
  orbit $x^*_{{\rm p.o.}}$ is linearly stable if (and only if) the same holds
  for the approximate periodic orbit $x^*$. In the unstable case, the
  number of hyperbolic directions of the periodic orbit $x^*_{{\rm p.o.}}$ is
  the same as for $x^*$.
\end{corollary}


\section{Normal form algorithm}
\label{sec:nfalg}

In this section, by using the formalism of Lie series (see
\cite{BenGGS84,GioLoc-1997}), we detail the generic step of the normal
form algorithm that takes the Hamiltonian at order $r$ and brings it
into normal form up to order $r+1$.  We stress here that the normal
form algorithm is a completely constructive procedure that can be
effectively implemented by means of computer algebra, see,
e.g.,~\cite{GioSan-2012}.

The relevant algebraic property of the $\Pset_\ell$ classes 
of functions is stated by the following

\begin{lemma}
  \label{lem:poisson}
  Let $f\in\Pscr_{s_1}$ and $g\in\Pscr_{s_2}$, then
  $\{f,g\}\in\Pscr_{s_1+s_2-2}$.
 \end{lemma}
The proof is left to the reader, being a trivial consequence
of the definition of the Poisson bracket.

\subsection{Generic r-th normalization step}

We summarize the five stages of a generic r-th normalizing step.
The starting Hamiltonian has the form
\begin{equation} \label{frm:H(r-1)}
\begin{aligned}
H^{(r-1)}&= \omega p_1 +\sum_{j=1}^{n_2} \imunit \Omega_j\xi_j\eta_j \\ 
                &\quad +\sum_{s<r} f_{0}^{(r-1,s)}
                +\sum_{s<r} f_{2}^{(r-1,s)}
                +\sum_{s<r} f_{3}^{(r-1,s)}+\sum_{s<r} f_{4}^{(r-1,s)}\\
                &\quad +f_{0}^{(r-1,r)} +f_{1}^{(r-1,r)} +f_{2}^{(r-1,r)}+f_{3}^{(r-1,r)}
                +f_{4}^{(r-1,r)} \\ 
                &\quad +\sum_{s>r} f_{0}^{(r-1,s)}
                +\sum_{s>r} f_{1}^{(r-1,s)} +\sum_{s>r} f_{2}^{(r-1,s)}
                +\sum_{s>r} f_{3}^{(r-1,s)}+\sum_{s>r} f_{4}^{(r-1,s)}\\            
                &\quad +\sum_{s\geq 0}\sum_{\ell>2}f_{\ell}^{(r-1,s)} \ .
\end{aligned}
\end{equation}
where $f_{0}^{(r-1,s)}$, $f_{2}^{(r-1,s)}$, $f_{3}^{(r-1,s)}$ and $f_{4}^{(r-1,s)}$,
for $1\leq s<r$, are in normal form. 

\subsubsection{First stage of the r-th normalization step}

We average the term $f_{0}^{(r-1,r)}$ with respect to the fast angle $q_1$, 
determining the generating function
$$
\chi_{0}^{(r)}(\hat{q})= X^{(r)}_{0}(\hat{q})+
\langle\zeta^{(r)},\hat{q}\rangle  \qquad \text{with} \qquad \zeta^{(r)}\in\reali^{n_1} \ ,
$$
belonging to $\Pset_0$ and of order $\Oscr(\epsilon^r)$, by solving
the homological equations
\begin{equation*}
  \begin{aligned}
   & \lie{X^{(r)}_{0}} \omega p_1
    + f_{0}^{(r-1,r)}  =\langle f_{0}^{(r-1,r)} \rangle_{q_1} \ ,\\
   & \lie{\prsca{\zeta^{(r)}}{{\hat q}}} f_{4}^{(0,0)}\Bigr|_{\xi=\eta=0} +
    \Bigl\langle
     f_{2}^{(r-1,r)}\Bigr|_{\xi=\eta=0\atop
    q=q^*}\Bigr\rangle_{q_1}
    = 0\ .
  \end{aligned}
\end{equation*}
By considering the Taylor-Fourier expansion
$$
 f_{0}^{(r-1,r)}(\hat{q})= \sum_{k} c_{0,0,0,k}^{(r-1,r)}\exp(\imunit
\langle k, \, \hat{q} \rangle) \ ,
$$
we obtain
$$
X^{(r)}_{0}(\hat{q}) =\sum_{k_1\neq0} \frac{c_{0,0,0,k}^{(r-1,r)}}{\imunit
  k_1\omega} \exp(\imunit \langle k, \, \hat{q} \rangle) \ .
$$
The vector $\zeta^{(r)}$ is determined by solving the 
linear system  
\begin{equation}
\sum_j C_{0,ij} \zeta_j^{(r)} =
\parder{}{{\hat p_i}}\Bigl\langle f_{2}^{(r-1,r)}
\Bigr|_{\xi=\eta=0 \atop q=q^*}\Bigr\rangle_{q_1} \ .
\label{frm:traslazioneR}
\end{equation}

The transformed Hamiltonian is computed as
\begin{equation}
\begin{aligned} \label{Ham-Istage-r-step}
H^{(\rmI;r-1)}&= \exp\Bigl(L_{\chi^{(r)}_{0}}\Bigr) H^{(r-1)}= \\ 
                &= \omega p_1 +\sum_{j=1}^{n_2} \imunit \Omega_j\xi_j\eta_j \\ 
                &\quad +\sum_{s<r} f_{0}^{(\rmI;r-1,s)}
                +\sum_{s<r} f_{2}^{(\rmI;r-1,s)}
                +\sum_{s<r} f_{3}^{(\rmI;r-1,s)} +\sum_{s<r} f_{4}^{(\rmI;r-1,s)}\\             
                &\quad +f_{0}^{(\rmI;r-1,r)} +f_{1}^{(\rmI;r-1,r)}
                +f_{2}^{(\rmI;r-1,r)} +f_{3}^{(\rmI;r-1,r)} +f_{4}^{(\rmI;r-1,r)} \\ 
                &\quad +\sum_{s>r} f_{0}^{(\rmI;r-1,s)}
                +\sum_{s>r} f_{1}^{(\rmI;r-1,s)} +\sum_{s>r} f_{2}^{(\rmI;r-1,s)}
                +\sum_{s>r} f_{3}^{(\rmI;r-1,s)} +\sum_{s>r} f_{4}^{(\rmI;r-1,s)}\\            
                &\quad +\sum_{s\geq 0}\sum_{\ell>2}f_{\ell}^{(\rmI;r-1,s)} \ .
\end{aligned} 
\end{equation} 
The functions $f_{\ell}^{(\rmI;r-1,s)}$ are recursively defined as
\begin{equation} \label{fzn-ricorsive-Istage-r-step}
\begin{aligned}
f_{0}^{(\rmI;r-1,r)} &= \langle f_{0}^{(r-1,r)} \rangle_{q_1} \ , \\
f_{\ell}^{(\rmI;r-1,s)}&= \sum_{j=0}^{\lfloor s/r \rfloor} \frac{1}{j!}
                        L_{\chi^{(r)}_{0}}^{j} f^{(r-1,s-jr)}_{\ell+2j} \ ,
                  \qquad\qquad\hbox{for }{\vtop{\hbox{$\ell=0,\ s\neq r$\ ,}
                  \vskip-2pt\hbox{\hskip-1pt or $\ell\neq0 \ s\geq 0$ \ ,}}}
\end{aligned}
\end{equation}
with $f_{\ell}^{(\rmI;r-1,s)} \in \Pscr_{\ell}$.

\subsubsection{Second stage of the r-th normalization step}

We now remove the term $f_{1}^{(\rmI;r-1,r)}$ by means of
the generating function $\chi^{(r)}_{1}$, belonging 
to $\Pset_1$ and of order $\Oscr(\epsilon^r)$, by solving the homological 
equation
\begin{equation} \label{omologica-IIstage-r-step}
L_{\chi^{(r)}_{1}}\Bigl({\omega}{p_1}+ \sum_{j=1}^{n_2} \imunit\Omega_j \xi_{j} \eta_{j}\Bigr)
+f_{1}^{(\rmI;r-1,r)}= 0 \ .
\end{equation}
Considering again the Taylor-Fourier expansion 
$$
f_{1}^{(\rmI;r-1,r)}(\hat{q},\xi,\eta)=\sum_{|m_1|+|m_2|=1\atop
  k}\, c_{0,m_1,m_2,k}^{(\rmI;r-1,r)} \exp(\imunit \langle k, \, \hat{q} \rangle
)\xi^{m_1} \eta^{m_2} \ ,
$$ 
we get
$$
\chi^{(r)}_{1}(\hat{q},\xi,\eta)=\sum_{|m_1|+|m_2|=1\atop
  k}\, \,\frac{c_{0,m_1,m_2,k}^{(\rmI;r-1,r)}\,\exp(\imunit
  \langle k, \, \hat{q} \rangle)\,\xi^{m_1} \eta^{m_2} }{\imunit
  \big[k_1\omega+ \langle m_1-m_2, \, \Omega\rangle\big]} \ .
$$
with $\Omega\in\reali^{n_2}$.

The transformed Hamiltonian is calculated as
\begin{equation}
\begin{aligned} \label{Ham-IIstage-r-step}
H^{(\rmII;r-1)} &= \exp\left(L_{\chi^{(r)}_{1}}\right)H^{(\rmI;r-1)}= \\
              &= \omega p_1 +\sum_{j=1}^{n_2} \imunit \Omega_j\xi_j\eta_j \\ 
                &\quad +\sum_{s<r} f_{0}^{(\rmII;r-1,s)}
                +\sum_{s<r} f_{2}^{(\rmII;r-1,s)}
                +\sum_{s<r} f_{3}^{(\rmII;r-1,s)} +\sum_{s<r} f_{4}^{(\rmII;r-1,s)}\\
                &\quad +f_{0}^{(\rmII;r-1,r)}
                +f_{2}^{(\rmII;r-1,r)} +f_{3}^{(\rmII;r-1,r)} +f_{4}^{(\rmII;r-1,r)} \\ 
                &\quad +\sum_{s>r} f_{0}^{(\rmII;r-1,s)}
                +\sum_{s>r} f_{1}^{(\rmII;r-1,s)} +\sum_{s>r} f_{2}^{(\rmII;r-1,s)}
                +\sum_{s>r} f_{3}^{(\rmII;r-1,s)} +\sum_{s>r} f_{4}^{(\rmII;r-1,s)}\\            
                &\quad +\sum_{s\geq 0}\sum_{\ell>2}f_{\ell}^{(\rmII;r-1,s)} \ ,
\end{aligned}
\end{equation}
with
\begin{equation} \label{fzn-ricorsive-IIstage-r-step}
\begin{aligned}
f_{1}^{(\rmII;r-1,r)} &= 0 \ , \\ 
f_{0}^{(\rmII;r-1,2r)} &= f_{0}^{(\rmI;r-1,2r)} 
                       + L_{\chi^{(r)}_{1}}f_{1}^{(\rmI;r-1,r)} 
                       +\frac{1}{2} L_{\chi^{(r)}_{1}} \left( L_{\chi^{(r)}_{1}}
                       f_{2}^{(\rmI;r-1,0)} \right)= \\ 
                       &= f_{0}^{(\rmI;r-1,2r)} +
                       \frac{1}{2} L_{\chi^{(r)}_{1}}f_{1}^{(\rmI;r-1,r)} \ , \\ 
f_{\ell}^{(\rmII;r-1,s)} &= \sum_{j=0}^{\lfloor s/r \rfloor} \frac{1}{j!}
                       L_{\chi^{(r)}_{1}}^{j} f^{(\rmI;r-1,s-jr)}_{\ell+j},
              \qquad\qquad\hbox{for }{\vtop{\hbox{$\ell=0,\ s\neq 2r$\ ,}
                  \vskip-2pt\hbox{\hskip-1pt or $\ell=1 \ s\neq r$ \ ,}
                  \vskip-2pt\hbox{\hskip-1pt or $\ell\geq2 \ s\geq 0$ \ ,}}}
\end{aligned}
\end{equation}
where we have exploited~\eqref{omologica-IIstage-r-step}.

\subsubsection{Third stage of the r-th normalization step}

We now average the term $f_{2}^{(\rmII;r-1,r)}$
with respect to the fast angle $q_1$,
determining the generating function $\chi^{(r)}_{2}$, belonging 
to $\Pset_2$ and of order $\Oscr(\epsilon^r)$, by solving the homological 
equation
\begin{equation} \label{omologica-IIIstage-r-step}
L_{\chi^{(r)}_{2}} \Bigl({\omega}{p_1}+ \sum_{j=1}^{n_2} \imunit\Omega_j \xi_{j} \eta_{j}\Bigr) +
f_{2}^{(\rmII;r-1,r)} = \langle
f_{2}^{(\rmII;r-1,r)}\rangle_{q_1} \ .
\end{equation}
Therefore, considering the Taylor-Fourier expansion
$$
f_{2}^{(\rmII;r-1,r)}(\hat{p},\hat{q},\xi,\eta)
=\sum_{|l|=1\atop k}\, c_{l,0,0,k}^{(\rmII;r-1,r)}\hat{p}^{l}\exp(\imunit
\langle k, \, \hat{q} \rangle) +\sum_{|m_1|+|m_2|=2\atop
  k}\, c_{0,m_1,m_2,k}^{(\rmII;r-1,r)} \exp(\imunit \langle k, \, \hat{q} \rangle
)\xi^{m_1} \eta^{m_2} \ ,
$$
we obtain
$$
 \chi_{2}^{(r)}(\hat{p},\hat{q},\xi,\eta) =\sum_{|l|=1\atop k_1\neq0}\,
\frac{c_{l,0,0,k}^{(\rmII;r-1,r)}\hat{p}^{l}\exp(\imunit
  \langle k, \, \hat{q} \rangle)}{\imunit k_1 \omega}+
  \sum_{|m_1|+|m_2|=2\atop
  k_1\neq 0}\, \,\frac{c_{0,m_1,m_2,k}^{(\rmII;r-1,r)}\,\exp(\imunit
  \langle k, \, \hat{q} \rangle)\,\xi^{m_1} \eta^{m_2} }{\imunit
  \big[k_1\omega+ \langle m_1-m_2, \, \Omega\rangle\big]} \ .
$$
The transformed Hamiltonian is computed as
\begin{equation*}
H^{(\rmIII;r-1)} = \exp\left(L_{\chi_{2}^{(r)}}\right)H^{(\rmII;r-1)}
\end{equation*}
and is in the form~\eqref{Ham-IIstage-r-step}, replacing 
the upper index II by III, with $f_{\ell}^{(\rmIII;r-1,s)}\in\Pset_\ell$ given by
\begin{equation} \label{fzn-ricorsive-IIIstage-r-step}
\begin{aligned}
f_{2}^{(\rmIII;r-1,r)} &= \langle f_{2}^{(\rmII;r-1,r)}\rangle_{q_1} \ , \\
f_{2}^{(\rmIII;r-1,ri)} &= \frac{1}{(i-1)!}L_{\chi^{(r)}_{2}}^{i-1} \left(
               f_{2}^{(\rmII;r-1,r)} +\frac{1}{i} L_{\chi^{(r)}_{2}}
               f^{(\rmII;r-1,0)}_{2} \right) \\ 
               &\quad+ \sum_{j=0}^{i-2} \frac{1}{j!} L_{\chi^{(r)}_{2}}^{j}
                f^{(\rmII;r-1,ri-rj)}_{2}  \\ 
                &= \frac{1}{(i-1)!}L_{\chi^{(r)}_{2}}^{i-1} \left(
                 \frac{1}{i}\langle f_{2}^{(\rmII;r-1,r)}\rangle_{q_1}
                 +\frac{i-1}{i}f^{(\rmII;r-1,r)}_{2}\right) \\ 
                 &\quad+ \sum_{j=0}^{i-2} \frac{1}{j!}
                  L_{\chi^{(r)}_{2}}^{j} f^{(\rmII;r-1,ri-rj)}_{2} \ , \\ 
f_{\ell}^{(\rmIII;r-1,s)} &= \sum_{j=0}^{\lfloor s/r \rfloor} \frac{1}{j!} L_{\chi^{(r)}_{2}}^{j}
                  f^{(\rmII;r-1,s-jr)}_{\ell} \ ,
                  \qquad\qquad\hbox{for }{\vtop{\hbox{$\ell=2,\ s\neq ri$\ ,}
                  \vskip-2pt\hbox{\hskip-1pt or $\ell\neq2, \ s\geq 0$ \ ,}}}
\end{aligned}
\end{equation}
where we have used the homological equation~\eqref{omologica-IIIstage-r-step}.

\subsubsection{Fourth stage of the r-th normalization step}

We now remove the term 
${f}_{1,1}^{(\rmIII;r-1,r)}\in\Pscr_{3}$,
which depends both on the actions
and on the transverse variables $\xi$, $\eta$.
We determine the generating function $\chi^{(r)}_{3}$, belonging 
to $\Pset_3$ and of order $\Oscr(\epsilon^r)$, by solving the homological 
equation
\begin{equation} \label{omologica-IVstage-r-step}
L_{\chi^{(r)}_{3}}\Bigl({\omega}{p_1}+ \sum_{j=1}^{n_2} \imunit\Omega_j
\xi_{j} \eta_{j}\Bigr) +{f}_{1,1}^{(\rmIII;r-1,r)}= 0 \ .
\end{equation}
Hence, considering the Taylor-Fourier expansion 
$$
{f}_{1,1}^{(\rmIII;r-1,r)}(\hat{q},\hat{p},\xi,\eta)=\sum_{{|l|=1\atop |m_1|+|m_2|=1}\atop
  k}\, c_{l,m_1,m_2,k}^{(\rmIII;r-1,r)} \exp(\imunit \langle k, \, \hat{q} \rangle
)\hat{p}^l \xi^{m_1} \eta^{m_2} \ ,
$$ 
we get
$$
\chi^{(r)}_{3}(\hat{q},\hat{p},\xi,\eta)=\sum_{{|l|=1\atop |m_1|+|m_2|=1}\atop
  k}\, \,\frac{c_{l,m_1,m_2,k}^{(\rmIII;r-1,r)}\,\exp(\imunit
  \langle k, \, \hat{q} \rangle)\,\hat{p}^l\xi^{m_1} \eta^{m_2} }{\imunit
  \big[k_1\omega+ \langle m_1-m_2, \, \Omega\rangle\big]} \ .
$$
with $\Omega\in\reali^{n_2}$.

The transformed Hamiltonian is computed as
$$
H^{(\rmIV;r-1)} = \exp\left(L_{\chi^{(r)}_{3}}\right)
H^{(\rmIII;r-1)} 
$$ 
and is given in the form~\eqref{Ham-IIstage-r-step}, 
replacing the upper index II by IV, with
\begin{equation} \label{fzn-ricorsive-IVstage-r-step}
\begin{aligned}
f_{3}^{(\rmIV;r-1,r)} &= f_{3}^{(\rmIII;r-1,r)}\Bigr|_{\hat{p}=0} \ , \\
f_{4}^{(\rmIV;r-1,2r)} &= f_{4}^{(\rmIII;r-1,2r)} + L_{\chi^{(r)}_{3}} f_{3}^{(\rmIII;r-1,r)}+
                       \frac{1}{2} L_{\chi^{(r)}_{3}}^2 f_{2}^{(\rmIII;r-1,0)}= \\
                    &= f_{4}^{(\rmIII;r-1,2r)} 
                      + \frac{1}{2}L_{\chi^{(r)}_{3}} f_{3}^{(\rmIII;r-1,r)}
                      +\frac{1}{2}L_{\chi^{(r)}_{3}} f_{3}^{(\rmIII;r-1,2)}\Bigr|_{\hat{p}=0} \ , \\
f_{\ell}^{(\rmIV;r-1,s)} &= \sum_{j=0}^{\lfloor s/r \rfloor} \frac{1}{j!}
                       L_{\chi^{(r)}_{3}}^{j} f^{(\rmIII;r-1,s-jr)}_{\ell-j},
    \qquad\qquad\hbox{for }{\vtop{\hbox{$\ell=3,\ s\neq r$ ,}
    \vskip-2pt\hbox{\hskip-1pt or $\ell=4,\ s\neq 2r$ ,}
    \vskip-2pt\hbox{\hskip-1pt or $\ell\neq3,4,\ s\geq0$ .}}}
\end{aligned}
\end{equation}
where we have exploited~\eqref{omologica-IVstage-r-step}.

\subsubsection{Fifth stage of the r-th normalization step}

We average the term $f_{4}^{(\rmIV;r-1,r)}\Bigr|_{\xi=\eta=0}$
with respect to the fast angle $q_1$.
We determine the generating function $\chi^{(r)}_{4}$, belonging 
to $\Pset_4$ and of order $\Oscr(\epsilon^r)$, by solving the homological 
equation
\begin{equation} \label{omologica-Vstage-r-step}
L_{\chi^{(r)}_{4}} {\omega}{p_1}+ 
f_{4}^{(\rmIV;r-1,r)}\Bigr|_{\xi=\eta=0} = \langle
f_{4}^{(\rmIV;r-1,r)}\Bigr|_{\xi=\eta=0}\rangle_{q_1} \ .
\end{equation}
By considering the Taylor-Fourier expansion
$$
f_{4}^{(\rmIV;r-1,r)}(\hat{p},\hat{q})
=\sum_{|l|=2\atop k}\, c_{l,0,0,k}^{(\rmIV;r-1,r)}\hat{p}^{l}\exp(\imunit
\langle k, \, \hat{q} \rangle) \ ,
$$
we obtain
$$
 \chi_{4}^{(r)}(\hat{p},\hat{q}) =\sum_{|l|=2\atop k_1\neq0}\,
\frac{c_{l,0,0,k}^{(\rmIV;r-1,r)}\hat{p}^{l}\exp(\imunit
  \langle k, \, \hat{q} \rangle)}{\imunit k_1 \omega} \ .
$$
The transformed Hamiltonian is calculated as
\begin{equation*}
H^{(r)} = \exp\left(L_{\chi_{4}^{(r)}}\right)H^{(\rmIV;r-1)}
\end{equation*}
and is given in the form~\eqref{frm:H(r-1)}, replacing the upper index $r-1$ by $r$,
with
\begin{equation} \label{fzn-ricorsive-Vstage-r-step}
\begin{aligned}
f_{4}^{(r,r)} &= \langle f_{4}^{(\rmIV;r-1,r)}\Bigr|_{\xi=\eta=0}\rangle_{q_1} 
                   + \left(f_{4}^{(\rmIV;r-1,r)}-f_{4}^{(\rmIV;r-1,r)}\Bigr|_{\xi=\eta=0}\right) \ , \\
f_{\ell}^{(r,s)} &= \sum_{j=0}^{\lfloor s/r \rfloor} \frac{1}{j!} L_{\chi^{(r)}_{4}}^{j}
                  f^{(\rmIV;r-1,s-jr)}_{\ell -2j} \ .
    \qquad\qquad\hbox{for }{\vtop{\hbox{$\ell=4,\ s\neq r$ ,}
    \vskip-2pt\hbox{\hskip-1pt or $\ell\neq4,\ s\geq0$ .}}}
\end{aligned}
\end{equation}

Before closing this section we think it is worth to stress the
connection between Theorem~\ref{teo:forma-normale-r} and the normal
form structure.  In order to simply state a theorem about the
continuation of the periodic orbits, three stages of the normalization
step would be enough, the last one consisting in the average of the
term ${f}_{1,0}^{(\rmII;r-1,r)}(\hat{p},\hat{q})$ only.  With this
minimal normal form construction, the abstract result would require
assumptions on the eigenvalues of the whole matrix $N(\epsilon)$, not only
on the block ${N_{11}}(\epsilon)$; indeed, three stages would not be
enough to split the spectrum of the matrix $N(\epsilon)$ in the spectrum
of two diagonal blocks.
As a consequence, with the purpose of getting a more accessible
criterion for applications, it is necessary to perform at least a
fourth stage in the normalization step, which allows to remove the
term
$f_{1,1}^{(\rmIII;r-1,r)}$,
thus achieving the desired structure with null block on the
anti-diagonal.  Let us stress that the fourth stage does not need a
second Melnikov condition.  However, with four steps only, the matrix
of the linearized system is not independent of time, due to lack of
averaging of the terms ${f}_{0,2}^{(\rmII;r-1,r)}(\xi,\eta)$ and
${f}_{2,0}^{(\rmII;r-1,r)}(\hat{p},\hat{q})$, namely part of the
third stage and the fifth stage.  Therefore, we cannot easily deduce
the structure of the matrix $M(\epsilon)$, simply considering the
linearization of the relative equilibrium for $Z^{(r)}$.
Nevertheless, this allows to simplify the statement, giving a
criterion on the eigenvalues of the block ${N_{11}}(\epsilon)$ which can
represent an easier condition to be checked in applications.


\section{Analytic estimates}
\label{sec:analytics}

Our aim now is to turn the formal algorithm into a recursive scheme of estimates.
Prior to describing the main results, we must anticipate some useful
technical tools.

\subsection{Estimates for Poisson brackets and Lie series}\label{sbs:stima-serie-di-Lie}

We report here some basic Cauchy's estimates which will be needed to
bound the transformed Hamiltonian.  Since similar estimates have
already been presented, see, e.g.~\cite{PenSD18}, we decide to not
dwell on it and only include the statement of the Lemmas.

\begin{lemma}\label{lem:stima-derivata-Lie}
Let $d\in\reali$ such that $0<d<1$ and $g\in\Pset_{\ell}$ be an analytic function with bounded norm
$\|g\|_{1}$.  Then one has
\begin{equation*}
\left\|\parder{g}{\hat{p}_j}\right\|_{1-d}\leq
\frac{\left\|g\right\|_{1}}{d\rho}\ ,
\qquad
\left\|\parder{g}{\hat{q}_j}\right\|_{1-d}\leq
\frac{\left\|g\right\|_{1}}{e d \sigma}\ ,
\qquad
\left\|\parder{g}{\xi_j}\right\|_{1-d}\leq
\frac{\left\|g\right\|_{1}}{d R}\ ,
\qquad
\left\|\parder{g}{\eta_j}\right\|_{1-d}\leq
\frac{\left\|g\right\|_{1}}{d R}\ ,
\end{equation*}
\end{lemma}

\begin{lemma}\label{lem:stima-derivata-gen}
Let $d\in\reali$ such that $0<d<1$ and $j\geq1$.  Then one has
\begin{align}
\left\|\lie{\chi_0^{(r)}}^{j}f\right\|_{1-d-d^{\prime}}
&\leq\frac{j!}{e}
\left(
\frac{e\|X_0^{(r)}\|_{1-d'}}{{d^2\rho\sigma}} + \frac{e |\zeta^{(r)}|}{{d\rho}}
  \right)^{j}
  \|f\|_{1-d^{\prime}}\ ,
  \label{frm:stima-liechi0}
\\
\left\|\lie{\chi_1^{(r)}}^{j}f\right\|_{1-d-d^{\prime}}
&\leq\frac{j!}{e^2}
\left(
\frac{\|\chi_1^{(r)}\|_{1-d'}}{{d^2}}\left(\frac{e}{\rho\sigma}+\frac{e^2}{R^2}\right)
  \right)^{j}
  \|f\|_{1-d^{\prime}}\ ,
    \label{frm:stima-liechi1}
\\
\left\|\lie{\chi_2^{(r)}}^{j}f\right\|_{1-d-d^{\prime}}
&\leq\frac{j!}{e^2}
\left(
\frac{\|\chi_2^{(r)}\|_{1-d'}}{{d^2}}\left(\frac{2e}{\rho\sigma}+\frac{e^2}{R^2}\right)
  \right)^{j}
  \|f\|_{1-d^{\prime}}\ ,
    \label{frm:stima-liechi2}
\\
\left\|\lie{\chi_3^{(r)}}^{j}f\right\|_{1-d-d^{\prime}}
&\leq\frac{j!}{e^2}
\left(
\frac{\|\chi_3^{(r)}\|_{1-d'}}{{d^2}}\left(\frac{2e}{\rho\sigma}+\frac{e^2}{R^2}\right)
  \right)^{j}
  \|f\|_{1-d^{\prime}}\ ,
    \label{frm:stima-liechi3}
\\
\left\|\lie{\chi_4^{(r)}}^{j}f\right\|_{1-d-d^{\prime}}
&\leq\frac{j!}{e^2}
\left(
\frac{2e\|\chi_4^{(r)}\|_{1-d'}}{{d^2\rho\sigma}}
  \right)^{j}
  \|f\|_{1-d^{\prime}}\ ,
    \label{frm:stima-liechi4}
\end{align}
\end{lemma}

\subsection{Recursive scheme of estimates}

Having fixed $d\in\reali$, $0 < d \leq 1/4$, we consider a
sequence ${\delta}_{r\geq1}$ of positive real numbers satisfying
\begin{equation}
  \delta_{r+1} \leq \delta_r\ ,\quad \sum_{r\geq1} \delta_r \leq \frac{d}{5}\ ,
  \label{frm:delta_r}
\end{equation}
and a further sequence ${d}_{r\geq0}$ defined as
\begin{equation}
  d_0 = 0\ ,\quad d_r = d_{r-1} + 5\delta_r\ .
  \label{frm:d_r}
\end{equation}
This sequence allows to control the restrictions of the domain due to
the Cauchy's estimate.

The factors entered by the estimate of the norm of the Poisson brackets
are bounded by
\begin{equation}
  \label{frm:Xi-r}
  \Xi_r=\max\left(\frac{eE}{\alpha\delta_r^2\rho\sigma}+  \frac{eE}{4m\delta_r\rho^2},
  2+\frac{eE}{\alpha\delta_r^2\rho\sigma},
  \frac{E}{\alpha\delta_r^2}\left( \frac{2e}{\rho\sigma} +\frac{e^2}{R^2} \right)
  \right)
  \ ,
\end{equation}
with
$$
\alpha = \min_{k_1, j, l, k} \left(|\omega|, |k_1\omega \pm \Omega_j |\,, | k_1\omega \pm \Omega_l \pm \Omega_k| \right)\ ,
$$
that is strictly greater than zero in view of the the Melnikov conditions.

The number of terms in~\eqref{fzn-ricorsive-Istage-r-step},
\eqref{fzn-ricorsive-IIstage-r-step},
\eqref{fzn-ricorsive-IIIstage-r-step},
\eqref{fzn-ricorsive-IVstage-r-step} and~\eqref{fzn-ricorsive-Vstage-r-step} is controlled by the five
sequences
\begin{equation} \label{frm:nu-sequence}
\vcenter{\openup1\jot \halign{
    $\displaystyle\hfil#$&$\displaystyle{}#\hfil$&$\displaystyle#\hfil$\cr
    \nu_{0,s} &= 1 &\quad\hbox{for } s\geq 0\,, \cr
    \nu_{r,s}^{(\rmI)} &= \sum_{j=0}^{\lfloor s/r \rfloor}
    \nu_{r-1,r}^{j}\nu_{r-1,s-jr} &\quad\hbox{for } r\geq 1\,,\ s\geq
    0\,, \cr
    \nu_{r,s}^{(\rmII)} &= \sum_{j=0}^{\lfloor s/r \rfloor}
    (\nu_{r,r}^{(\rmI)})^{j}\nu_{r,s-jr}^{(\rmI)} 
    &\quad\hbox{for } r\geq 1\,,\ s\geq 0\,, \cr 
    \nu_{r,s}^{(\rmIII)} &= \sum_{j=0}^{\lfloor s/r \rfloor}
    (2\nu_{r,r}^{(\rmII)})^{j}\nu_{r,s-jr}^{(\rmII)}
    &\quad\hbox{for } r\geq 1\,,\ s\geq 0\,.  \cr 
   \nu_{r,s}^{(\rmIV)} &= \sum_{j=0}^{\lfloor s/r \rfloor}
   (\nu_{r,r}^{(\rmIII)})^{j}\nu_{r,s-jr}^{(\rmIII)}
   &\quad\hbox{for } r\geq 1\,,\ s\geq 0\,.  \cr 
   \nu_{r,s} &= \sum_{j=0}^{\lfloor s/r \rfloor}
   (\nu_{r,r}^{(\rmIV)})^{j}\nu_{r,s-jr}^{(\rmIV)}
   &\quad\hbox{for } r\geq 1\,,\ s\geq 0\,.  \cr }}
\end{equation}

\noindent
Again, since similar estimates has been already presented, see,
e.g.~\cite{PenSD18}, we only include the statement of the following Lemma.
 
\begin{lemma}\label{lem:nu}
The sequence of positive integers $\{\nu_{r,s}\}_{r\ge 0\,,\,s\ge 0}$
defined in~\eqref{frm:nu-sequence} is bounded by
$$
\nu_{r,s}\leq\nu_{s,s}\leq \frac{2^{14s}}{2^8}\ .
$$
\end{lemma}

The following Lemma collects all the key estimates concerning the
generating functions and the transformed Hamiltonians.  A detailed
proof of the Lemma will take several pages of straightforward (and
tedious) calculation.  Since the key aspects have already been
presented in~\cite{PenSD18} we omit the proof and leave the adaptation to the
willing reader.

\begin{lemma} \label{lem:lemmone.1}
  Consider a Hamiltonian $H^{(r-1)}$ expanded as in~\eqref{frm:H(0)}.
  Let $\chi_0^{(r)}$, $\chi_1^{(r)}$, $\chi_2^{(r)}$, $\chi_3^{(r)}$ 
  and $\chi_4^{(r)}$ be the generating functions
  used to put the Hamiltonian in normal form at order~$r$, then one has
  \begin{equation*}
  \begin{aligned}
    \|X_0^{(r)}\|_{1-d_{r-1}} &\leq \frac{1}{\alpha} \nu_{r-1,r} \Xi_r^{5r-5} E\epsilon^r\ ,\cr
    |\zeta^{(r)}| &\leq  \frac{1}{4m\rho} \nu_{r-1,r} \Xi^{5r-3}E\epsilon^r\ ,\cr
    \|\chi_1^{(r)}\|_{1-d_{r-1}-\delta_r} &\leq  \frac{1}{\alpha}
    \nu_{r,r}^{(\rmI)} \Xi_r^{5r-4}  \frac{E}{2}\epsilon^r  \ ,\cr
    \|\chi_2^{(r)}\|_{1-d_{r-1}-2\delta_r} &\leq \frac{1}{\alpha}
    2\nu_{r,r}^{(\rmII)}   \Xi_r^{5r-3}  \frac{E}{2^2}\epsilon^r\ ,\cr
    \|\chi_3^{(r)}\|_{1-d_{r-1}-3\delta_r} &\leq  \frac{1}{\alpha}
    \nu_{r,r}^{(\rmIII)} \Xi_r^{5r-2}  \frac{E}{2^3}\epsilon^r  \ ,\cr
    \|\chi_4^{(r)}\|_{1-d_{r-1}-4\delta_r} &\leq  \frac{1}{\alpha}
    \nu_{r,r}^{(\rmIV)} \Xi_r^{5r-1}  \frac{E}{2^4}\epsilon^r  \ .
  \end{aligned}
  \end{equation*}
  The terms appearing in the expansion of $H^{(r)}$, i.e.
  in~\eqref{frm:H(r)}, are bounded as
  \begin{equation}\label{frm:f_l^(r,s)}
  \begin{aligned}
  \|f_{\ell}^{(r,s)}\|_{1-d_r} &\leq \nu_{r,s} 
  \Xi_r^{5s} \frac{E}{2^\ell}\epsilon^s\ .
  \end{aligned}
  \end{equation}
\end{lemma}
Let us stress that the proof of the Lemma actually requires stricter
estimates in~\eqref{frm:f_l^(r,s)} both for the lower order terms, as
is evident from the bounds on the generating functions, and for the
intermediate stages of the $r$-th normalization step.  The
interested reader can refer to~\cite{PenSD18}, {\sl mutatis mutandis}.

\subsection{Estimates for the approximate periodic orbit}

\begin{lemma}\label{lem:stima-approx-p.o.}
Let $x^*=(q^*,0,0)$ be a relative equilibrium for the truncated normal
form $Z^{(r)}$ defined in~\eqref{frm:Z(r)}, then $x^*$ is an approximate
periodic orbit for the Hamiltonian $H^{(r)}$ of order
$\Oscr(\epsilon^{r})$. Precisely, there exist $\epsilon^*(r)$ and $c_1(r)$ such that for $\epsilon<\epsilon^*$
\begin{equation}
  \label{frm:est.approx}
\norm{\Upsilon(x^*;\epsilon,q_{1}(0))}\leq c_1\epsilon^{r+1}\ .
\end{equation}
\end{lemma}

\begin{proof}
Consider the remainder $H^{(r)}-Z^{(r)}$, namely
$\sum_{s>r}\sum_{\ell\geq 0} f_\ell^{(r,s)}$.  Then, for~$\epsilon$ small
enough, i.e., for $\epsilon< 1/({2^{14} \Xi_r^{5}})$, one has
\begin{displaymath}
\begin{aligned}
\sum_{s>r}\sum_{\ell\geq 0} \Vert f_\ell^{(r,s)}\Vert \leq
\sum_{s>r}\sum_{\ell\geq 0} 2^{14s} \Xi_r^{5s}\frac{E}{2^\ell}\epsilon^s \leq
2E\frac{\left( 2^{14} \Xi_r^{5}\epsilon\right)^{r+1}}{1-2^{14} \Xi_r^{5}\epsilon}
\end{aligned}\ ,
\end{displaymath}
where we used the estimates in Lemma~\ref{lem:lemmone.1}.  Applying
the Cauchy's estimate for the symplectic gradient and integrating over
the period $T$, we can deduce that there exist a domain $\Uscr$ and a
constant $c_1(r)$, dependent on the domain, such that
\begin{displaymath}
  \norm{\Upsilon(x^*;\epsilon,q_1(0))}\leq c_1(r) \epsilon^{r+1}\ ,
\end{displaymath}
i.e., for $\epsilon^*(r)= 1/(2^{14}\Xi_r^{5})$, one can take
$c_1(r) = 4E \tond{2^{14} \Xi_r^{5}}^{r+1}$.
\end{proof}


\section{The \emph{seagull} example}
\label{s:appl}

In this Section we study in detail the pedagogical
example~\eqref{e.ex.dnls} already presented in the Introduction,
focusing on the continuation of degenerate periodic orbits via the
normal form construction described in the paper. Although the present
example does not represent a dNLS lattice in the proper sense due to
the limited number of sites, it is nevertheless suitable to see the
benefits of our normal form construction in a direct and easy way,
thanks to the simplicity of the change of coordinates between
Cartesian and action-angles. Furthermore, it sheds some light onto the
role of the nonlinearity in the linear stability of multi-peaked
discrete solitons in dNLS lattices.  Indeed, at variance with models
considered in literature, we notice that a change in the sign of the
nonlinear parameter $\gamma$ does not influence the nature of the
degenerate eigenspaces.  On the contrary, considering consecutive
excited sites as in~\cite{PelKF05}, a change in the sign of $\gamma$
(at fixed linear interaction $\epsilon$) produces an exchange of
stable and unstable degenerate directions around the periodic
solutions; actually the switch from saddle to center depends on the
sign of the product $\gamma\epsilon$.

Let us remark that, in order to apply
Theorem~\ref{teo:forma-normale-r}, we have to control the smallest
eigenvalue of the matrix $M(\epsilon)$, see~\eqref{e.small.eig}.  This
is a delicate point, particularly in actual applications.  Indeed, to
numerically verify this assumption one has to investigate the spectrum
of the matrix $\tond{\exp\tond{L_{11}(\epsilon)T}-\Id}_{\flat}$, by
interpolating the decay of the smallest eigenvalue with respect to $\epsilon$.

However, in some specific cases like the example here considered, one
can further decompose the quadratic Hamiltonian $Z_2^{(r)}$ in~\eqref{frm:square-K} in order to decouple the fast variables
$(Q_1,P_1)$ from the slow variables $(Q,P)$ with a linear canonical
change of coordinates (see~\cite{Tre91}).

Precisely, we decompose the matrix $C$ so as to put in evidence the
first row and column vectors, namely
\begin{displaymath}
 C=\begin{pmatrix} C_{11} & C_{12}\\ C_{21} & C_{22}
  \end{pmatrix}
\end{displaymath}
where $C_{11}$ is the first element, $C_{12} = C_{21}^\top$ is the
$(n_1-1)$-dimensional row vector and $C_{22}$ is the
$(n_1-1)$-dimensional square matrix.

Assume now that $C_{22}$ is invertible, then we can introduce the
canonical change of coordinates
\begin{equation}
  \label{e.can.lin}
  u = Q\ ,\quad u_1 = Q_1 - C_{12}C_{22}^{-1}Q\ , \quad
  P = v - v_1C_{22}^{-1}C_{21}\ , \quad P_1 = v_1\ .
\end{equation}
The transformed quadratic Hamiltonian $Z^{(r)}_2$ in~\eqref{frm:square-K} now reads
\begin{equation}
  \label{e.Zquadr.spl}
  Z^{(r)}_2 =  \frac{1}2 c_{11} v_1^2 + \frac{1}2
    \quadr{\langle Bu,u\rangle + \langle C_{22}v,v\rangle} + \langle Du, v\rangle+ \frac12 \langle G\xi,\xi\rangle + \langle E\xi,\eta\rangle + \frac12 \langle F\eta,\eta\rangle\ ,
\end{equation}
where $c_{11}= C_{11}-C_{12}C_{22}^{-1}C_{21}$ and the term $\langle Du, v\rangle$ contains mixed terms in action-angles variables.  The main
advantage is that, if $D=0$, then the fast dynamics and the slow one
turn out to be decoupled, hence it suffices to investigate the
eigenvalues of the matrix
\begin{equation}
\begin{pmatrix}
    0 & C_{22}\\
    -B& 0
\end{pmatrix}
\label{frm:matls}
\end{equation}
which represents the linear vector fields of the new slow variables
$(Q,P)$.  Hence~\eqref{e.small.eig} can be easily checked, possibly
without the needs of numerical interpolation.

Consider the Hamiltonian system~\eqref{e.ex.dnls}, that we report here
for convenience
\begin{displaymath}
 H = H_{0}+\varepsilon H_{1} = \sum_{j=-3}^{3}\Biggl( \frac{x_j^2 +
  y_j^2}{2} + \gamma\Biggl( \frac{x_j^2 + y_j^2}{2}\Biggr)^2 \Biggr) +\epsilon 
  \sum_{j=-3}^{2} 
(x_{j+1}x_j+ y_{j+1}y_j) \ ,
\end{displaymath}
with fixed boundary conditions $x_{-3}=y_{-3}=x_{3}=y_{3}=0$.

Following the procedure reported in the Introduction, we introduce
action-angle variables $ (x_j,y_j) =
(\sqrt{2I_j}\cos\varphi_j,-\sqrt{2I_j}\sin\varphi_j)$, for the set of
indices $j\in\Iscr=\{-2,-1,1,2\}$, and complex canonical coordinates
for the remaining central one $(x_0,y_0)$.  Then we focus on the
4-dimensional resonant torus $I=I^*$ with ${I}_j^*=I_l^*$, for
$j,l\in\Iscr$, and $\xi_0=\eta_0=0$.  Finally, introducing the
canonical change of coordinates~\eqref{e.4sites.tr} we get
\begin{displaymath}
\begin{aligned}
H^{(0)} =\, &\omega p_1 + \imunit\xi_0\eta_0 + f_4^{(0,0)}
\\
&+ f_0^{(0,1)}+ f_1^{(0,1)}+ f_2^{(0,1)}+ f_3^{(0,1)} +
 f_4^{(0,1)} + \sum_{\ell>4} f_\ell^{(0,1)} + \Oscr(\epsilon^2)\ ,
\end{aligned}
\end{displaymath}
with $\omega=1+2\gamma I^*$ and
$$
\begin{aligned}
f_4^{(0,0)}(\hat{p},\xi_0,\eta_0) &= \gamma\left((p_1 -p_2)^2 +  (p_2 -p_3)^2 
+ (p_3 -p_4)^2 +p_4^2\right) -\gamma\xi_0^2\eta_0^2 \ , \\
 f_0^{(0,1)}(q_2,q_4) &=\epsilon\left(2I^* \cos(q_2)+ 2I^* \cos(q_4)\right)\ ,\\
 f_1^{(0,1)}(\hat{q},\xi_0,\eta_0) &= \epsilon\biggl( (\xi_0+\imunit\eta_0)\sqrt{I^*} \cos(q_1+q_2)
 + (\xi_0+\imunit\eta_0) \sqrt{I^*} \cos(q_1+q_2+q_3)+ \\
 & \qquad- \imunit(\xi_0-\imunit\eta_0)\sqrt{I^*} \sin(q_1+q_2) 
 - \imunit(\xi_0-\imunit\eta_0)\sqrt{I^*} \sin(q_1+q_2+q_3)\biggr)\ , \\
  f_2^{(0,1)}(\hat{p},q) &= \epsilon\left( (p_1-p_3)\cos(q_2) 
  + p_3\cos(q_4)\right) \ , \\
f_3^{(0,1)}(p,\hat{q},\xi_0,\eta_0) &= \epsilon\biggl(\frac{(p_2-p_3) (\xi_0+\imunit\eta_0)\cos(q_1+q_2)}{2 \sqrt{I^*}} 
-\frac{\imunit (p_2-p_3) (\xi_0-\imunit\eta_0) \sin(q_1+q_2)}{2 \sqrt{I^*}}+ \\
& \qquad +\frac{(p_3-p_4) (\xi_0+\imunit\eta_0) \cos(q_1+q_2+q_3)}{2 \sqrt{I^*}}
-\frac{\imunit (p_3-p_4) (\xi_0-\imunit\eta_0) \sin(q_1+q_2+q_3)}{2 \sqrt{I^*}}\biggr) \ , \\
f_4^{(0,1)}(p,q) &= \epsilon\biggl( -\frac{((p_1-p_2)^2+ (p_2-p_3)^2) \cos(q_2)}{4 I^*}
+\frac{(p_1-p_2) (p_2-p_3) \cos(q_2)}{2 I^*} +\\
& \qquad -\frac{((p_3-p_4)^2+p_4^2) \cos(q_4)}{4 I^*}
+\frac{p_4(p_3-p_4) \cos(q_4)}{2 I^*} \biggr)\ .
\end{aligned}
$$

The transformed Hamiltonian, $H^{(0)}$, is now in a suitable form for
applying our normal form procedure. As already said in the
Introduction, in order to remove the degeneracy, we need to compute
the normal form up to order two.  Hence, in the following, we detail
all the needed normal form steps.

Since $f_0^{(0,1)}$ does not depend on $q_1$, the first stage of the
first normalization step only consists in the translation of the
actions $\hat p$ which allows to keep the frequency fixed.  Moreover,
as $f_2^{(0,1)}$ does not depend on the fast angle $q_1$, the
equation for $\zeta^{(1)}$ reads
$$ \sum_j C_{0,i,j} 
{\zeta}_j^{(1)} =\frac{\partial}{\partial \hat {p_i}} f_{2}^{(0,1)} \Bigr|_{{q}={q}^*}\ ,
$$
with
$$
C_0=\gamma\left(
\begin{array}{cccc}
2 & -2 & 0 & 0 \\
-2 & 4 & -2 & 0 \\
0 & -2 & 4 & -2 \\
0 & 0 & -2 & 4
\end{array}
\right)
$$
and the solution is
\begin{equation} \label{zeta1}
\zeta^{(1)}= \left(\frac{\epsilon}{\gamma}\left(\cos(q_2^*)+ \cos(q_4^*)\right)\ ,\ 
\frac{\epsilon}{\gamma}\left(\frac{\cos(q_2^*)}{2}+ \cos(q_4^*)\right) \ ,\ 
\frac{\epsilon}{\gamma}\cos(q_4^*) \ ,\ 
\frac{\epsilon}{2\gamma}\cos(q_4^*)\right)\ .
\end{equation}

The generating function $\chi_1^{(1)}$ that kills the term
$f_1^{(\rmI;0,1)}= f_1^{(0,1)}$ reads
\begin{displaymath}
\begin{aligned}
\chi_1^{(1)} &=
 \epsilon\biggl( \imunit\frac{ \sqrt{I^*} \left(e^{-\imunit (q_1+ q_2)}
 +e^{-\imunit (q_1+ q_2+ q_3)}\right) \xi _0}{\omega -1}
 +\frac{\sqrt{I^*} \left(e^{\imunit (q_1+ q_2)}
 +e^{\imunit (q_1+ q_2+ q_3)}\right)\eta _0}{\omega-1}
\biggr)\ .
\end{aligned}
\end{displaymath}

Since $f_{2}^{(\rmII;0,1)} = f_{2}^{(0,1)} +
L_{\prsca{\zeta^{(1)}}{{\hat q}}} f_{4}^{(0,0)} $ is again independent
of $q_1$, no further average is required in the third stage of the
normalization procedure.

Next, we have to remove the cubic terms which depend both on the
actions and on the transverse variables from $f_3^{(\rmIII,0,1)}=
f_3^{(0,1)}+ L_{\chi_1^{(1)}} f_{4}^{(0,0)}$.  This is done via the generating function
$$
\begin{aligned}
\chi_3^{(1)}=&-\epsilon e^{-\imunit \left(q_1+q_2+q_3\right)}\frac{\left(-1+e^{\imunit
    q_3}\right) p_3 \left(e^{\imunit \left(2 q_1+2 q_2+q_3\right)} \eta _0-\imunit
  \xi _0\right)+e^{\imunit q_3} p_2 \left(e^{2 \imunit \left(q_1+q_2\right)} \eta
  _0+\imunit \xi_0\right)}{2 \sqrt{I^*} \left( \omega-1 \right) }\\ &
+\epsilon e^{-\imunit \left(q_1+q_2+q_3\right)}\frac{p_4 \left(e^{2 \imunit
    \left(q_1+q_2+q_3\right)} \eta _0+\imunit \xi _0\right)}{2 \sqrt{I^*}
  \left( \omega-1 \right)}\ .
\end{aligned}
$$

Finally, the term $f_4^{(\rmIV,0,1)}= f_4^{(0,1)}$ turns out to be
independent of $q_1$, thus the first normalization step is concluded.

\smallskip

As already noticed, the solution of $\nabla_q f_0^{(1,1)}=0$, that
determines the $q^*$, and so the approximate periodic orbits, is
given by the four one-parameter families $Q_1=(0,\vartheta,0)$,
$Q_2=(0,\vartheta,\pi)$, $Q_3=(\pi,\vartheta,0)$,
$Q_4=(\pi,\vartheta,\pi)$, with $\vartheta\in S^1$. Thus, a further
normalization step is needed in order to investigate the continuation
of these families of periodic orbits.  The transformed Hamiltonian at
order two reads
\begin{displaymath}
\begin{aligned}
H^{(2)} =\, &\omega p_1 + \imunit\xi_0\eta_0 + f_4^{(2,0)} 
\\
&+ f_0^{(2,1)}+  f_2^{(2,1)}+  f_4^{(2,1)} + \sum_{\ell>4}  f_\ell^{(2,1)} \\
&+ f_0^{(2,2)}+  f_2^{(2,2)}+  f_4^{(2,2)} + \sum_{\ell>4}  f_\ell^{(2,2)} + \Oscr(\epsilon^3)\ .
\end{aligned}
\end{displaymath}

The approximate periodic orbits are the solutions $q^*$ of
\begin{equation}
  \label{e.2steps.qstar}
\nabla_q {f}_0^{(2,1)}({q}) + \nabla_q {f}_0^{(2,2)}({q}) = 0\ ,
\end{equation}
with, neglecting the constant terms, 
\begin{displaymath}
  {f}_0^{(2,2)}({q}) = L_{\prsca{\zeta^{(1)}}{{\hat q}}} f_{2}^{(0,1)}
  + \frac{1}{2}L_{\chi_1^{(1)}} f_{1}^{(0,1)} =
  \frac{\epsilon^2}{\gamma}\quadr{\cos(q_3)-\cos(q_2)\cos(q_2^*)-\cos(q_4)\cos(q_4^*)}\ .
\end{displaymath}
The system~\eqref{e.2steps.qstar} can be written as
\begin{equation} \label{F(q,epsilon)}
F(q,\epsilon)=F_0(q)+\epsilon F_1(q)=0
\end{equation}
where $F:\mathbb{T}^3 \times \mathcal{U}(0) \rightarrow \reali^3$ and
$F_0(Q_j(\vartheta)) = 0$.  We now introduce the matrices
$\tilde{B}_{0,j}(\vartheta)=\frac{\partial
  F_0(Q_j(\vartheta))}{\partial {q}}$ and remark that the tangent
direction to the four families $\partial_{\vartheta}Q_j = (0,1,0)$ is
the Kernel direction of $\tilde{B}_{0,j}(\vartheta)$, for
$j=1,\ldots,4$.  By computing
\begin{displaymath}
\inter{F_1(Q_j(\vartheta,0)),\partial_\theta Q_j}=
-\dfrac{ \sin(\vartheta)}{\gamma} \qquad j=1,\ldots,4  \ ,
\end{displaymath}
it is possible to deduce from standard bifurcation arguments (see
\cite{PenSPKK18,PenSD18}) that, apart from the \emph{in} and
\emph{out-of-phase} configurations $(0,0,0)$, $(0,\pi,0)$,
$(0,0,\pi)$, $(0,\pi,\pi)$, $(\pi,0,0)$, $(\pi,\pi,0)$, $(\pi,0,\pi)$
and $(\pi,\pi,\pi)$, the four families break down. In order to ensure
the continuation of these configurations, we have to verify the
condition~\eqref{e.small.eig}, with~\eqref{e.small.Ups} taking the
form
\begin{displaymath}
\norm{\Upsilon(x^*;\epsilon,q_1(0))}\leq c_1\epsilon^3\ .
\end{displaymath}

In order to verify~\eqref{e.small.eig} we have two options:
(i)~examine the spectrum of
$\tond{\exp\tond{L_{11}(\epsilon)T}-\Id}_{\flat}$ and numerically
interpolate the smallest eigenvalue, getting $|\lambda | \gtrsim \epsilon$
for each of the eight configuration; (ii)~since $q_j=q^*=\{0,\pi\}$,
the mixed terms in action-angle variables are missing and we can
compute $e^{T\sigma_l}-1 $, with $\sigma_l$ eigenvalues of the matrix
\begin{displaymath}
\begin{pmatrix}
0 & C_{22}\\
 -B & 0
\end{pmatrix} \ ,
\end{displaymath}
with
\begin{equation}
  \label{e.4sites.B}
B=  
  \begin{pmatrix}
    -2\epsilon I^*\cos(q_2)+\frac{\epsilon^2}{\gamma}\cos^2(q_2) & 0 &
    0\\
    0 & -\frac{\epsilon^2}{\gamma}\cos(q_3) & 0\\
    0 & 0 & -2\epsilon I^*\cos(q_4)+\frac{\epsilon^2}{\gamma}\cos^2(q_4)
  \end{pmatrix}\ .
\end{equation}
Hence, $C_{22}$ being definite and of order $\Oscr(1)$ in the limit of
small $\epsilon$, we obtain that condition~\eqref{e.small.eig} is plainly
verified with $\alpha=1$ and $r=2$. Applying the Theorem
\ref{teo:forma-normale-r} we can infer the existence of a unique
$x^*_{\rm p.o.}(\epsilon)=(q^*_{\rm p.o.}(\epsilon),\hat{p}_{\rm
  p.o.}(\epsilon),\xi_{\rm p.o.}(\epsilon),\eta_{\rm
  p.o.}(\epsilon))$, with $ q^*_{\rm p.o.}(\epsilon)=q^*=\{ 0, \pi \}$
such that $\norm{x^*_{\rm p.o.}-x^*}\leq c_0\epsilon^{2}$, for each
candidate for the continuation.

Coming to the linear stability, first we exploit the structure of the
quadratic Hamiltonian~\eqref{e.Zquadr.spl}, with $D\equiv 0$, in order
to get the approximate linear stability of the continued periodic
orbits.  It turns out that the stable and unstable directions
correspond to the positive or negative eigenvalues of $\gamma C_{22}
B$, where the prefactor $\gamma$ accounts for the positive or negative
signature of $C_{22}$ (which is the same of $C$). Hence,
the stability depends on the signature of $\gamma B$, given by the
elements on its diagonal
\begin{equation}
  \label{e.4sites.stab}
\gamma B=  
  \begin{pmatrix}
    -2\epsilon\gamma I^*\cos(q_2)+{\epsilon^2}\cos^2(q_2) & 0 &
    0\\
    0 & -{\epsilon^2}\cos(q_3) & 0\\
    0 & 0 & -2\epsilon\gamma I^*\cos(q_4)+{\epsilon^2}\cos^2(q_4)
  \end{pmatrix}\ .
\end{equation}
The degenerate direction depends only on $\epsilon^2$, thus it is always a
saddle at $q_3=0$ and always a center at $q_3=\pi$. Instead, the
nondegenerate directions depend on the sign of the product
$\gamma\epsilon$, which converts hyperbolic subspace into center subspace
at fixed $q_{2,4}\in\graff{0,\pi}$.  In particular, by studying the
spectrum in the in/out-of-phase configurations and for attractive
interactions ($\epsilon>0$), we find that the only stable approximate
periodic orbit corresponds to the configuration $(\pi,\pi,\pi)$ when
$\gamma>0$ and to $(0,\pi,0)$ when $\gamma<0$.  In order to derive the
effective linear stability, we have to verify~\eqref{e.dist.eig}.
Symbolic calculations implemented in Mathematica give
\begin{align*}
\lambda_{1,2}&=\pm\imunit\left(2 \sqrt{2} \sqrt{I^* |\gamma| } \sqrt{\epsilon}+\frac{ \sqrt{2} \epsilon^{3/2}}{\sqrt{I^* |\gamma | }}+h.o.t.\right)\\
\lambda_{3,4}&=\pm\imunit\left(2  \sqrt{2} \sqrt{I^* |\gamma| } \sqrt{\epsilon}+\frac{5  \epsilon^{3/2}}{2 \sqrt{2} \sqrt{I^* |\gamma| }}+h.o.t.\right)\\
\lambda_{5,6}&=\pm\imunit\left( \sqrt{2} \epsilon-\frac{ \epsilon^2}{4 \sqrt{2} I^* |\gamma| }+h.o.t.\right) \ ,\\
\end{align*}
thus condition~\eqref{e.dist.eig} holds true with $r+1-\alpha=2$
and $\beta=\frac{3}{2}$.


\begin{appendix}
\renewcommand{\thesection}{\Alph{section}}
  
\section{Spectrum deformation under matrix perturbations}
\label{s:app}

In this section we collect some useful results concerning the
deformation of the spectrum of a matrix under small perturbations.
The results are based on resolvent formalism, we refer to the
classical book of Kato~\cite{Kat76} for a detailed treatment of the
subject and to~\cite{MacS98,AhnMS01} for the study of the linear
stability of breathers and multibreathers.

We first set some notations. Given a matrix $M$ defined on a vector
space $X$, we denote by $\Sigma(M)$ its spectrum and by
$\rho(M)=\max_{\lambda\in\Sigma(M)}{|\lambda_j|}$ its spectral
radius. For any $z\not\in\Sigma(M)$ it is well defined
$R(z)=\tond{M-z}^{-1}$, which is the \emph{resolvent} of $M$. The
inverse of the spectral radius of $R(z)$ measures the distance between
$z\in\complessi$ and the spectrum of $M$
\begin{equation}
\label{e.inv.rho}
\dist\tond{z,\Sigma(M)} = \frac1{\rho(R(z))}\ .
\end{equation}
Let us recall that $\rho(M)$ and the
operatorial norm $\opnorm{M}=\sup_{x\neq
  0}\frac{\norm{Mx}}{\norm{x}}$ satisfy
\begin{equation}
\label{e.est.1}
\rho(M)\leq \opnorm{M}\ .
\end{equation}
Moreover, a converse inequality is given by the following
\begin{lemma}
\label{l.op.est}
Let $M:X\to X$, then there exists $c_{\textsf{op}}>1$, depending on the dimension of
$X$ only, such that
\begin{equation}
\label{e.est.2}
\opnorm{M}\leq c_{\textsf{op}}\max\graff{\rho(M),1}\ .
\end{equation}
If $M$ is diagonalizable, or if $\rho(M)\geq 1$, then the above simplifies
to
\begin{equation}
\label{e.est.2.bis}
\opnorm{M}\leq c_{\textsf{op}} \rho(M)\ .
\end{equation}
\end{lemma}

\subsubsection{On the minimum eigenvalue}

Let us now consider a given matrix $N$ and its perturbation
$M(\mu)=N + \mu P$ depending analytically on a small parameter
$\mu$.  The resolvent $R(z,\mu)$ of $M(\mu)$ is still well
defined and holomorphic in the two variables, provided
$z\not\in\Sigma(M(\mu))$; moreover, it is possible to relate the
perturbed and unperturbed resolvents via the following series
expansion
\begin{equation}
\label{e.second_Neum}
R(z,\mu) = R_0(z)\quadr{\Id+A(z,\mu)R_0(z)}^{-1}\ ,\quad
R_0(z)=R(z,0)\ ,
\end{equation}
where $R_0(z)$ is the resolvent of the leading term $N=M(0)$, while
$A(z,\mu)= R(z,\mu)-R_0(z)$ represents the \emph{deformation} due to
the small perturbation. In what follows we collect some useful results
relating the \emph{unperturbed spectrum} $\Sigma(N)$ to the
\emph{perturbed spectrum} $\Sigma(M(\mu))$.

We collect the results relating the minimum eigenvalue of $N$ and $M$
in the following

\begin{proposition}
\label{p.min.eigen}
Let us consider a matrix $M(\epsilon)=N(\epsilon)+\mu(\epsilon) P(\epsilon)$,
depending on the small parameter $\epsilon\in\Uscr(0)$. Let us assume that
for any $\epsilon\in\Uscr$ it holds true:
\begin{enumerate}[label=(\arabic*)]
\item $N(\epsilon)$ is invertible and there exist $c_1>0$ and $\alpha>0$
independent of $\epsilon$ such that
\begin{displaymath}
\opnorm{N^{-1}}\leq c_1 |\epsilon|^{-\alpha} \ ;
\end{displaymath}

\item $P(\epsilon)=P(0)+\Oscr(\epsilon)$ and there exist $c_2>0$ and
  $\beta>\alpha$ such that
\begin{displaymath}
|\mu(\epsilon)|\leq c_2 |\epsilon|^{\beta}\ .
\end{displaymath}
\end{enumerate}
Then for $\epsilon$ small enough $M(\epsilon)$ is invertible and there exists
$c_3>0$ independent of $\epsilon$ such that
\begin{equation}
\label{e.bound.min.eig}
|\nu|\geq c_3 |\epsilon|^\alpha\ ,\qquad\hbox{for all}\quad \nu\in\Sigma(M).
\end{equation}
Moreover, the same result holds true if we replace the first
assumption with
\begin{equation}
  \label{e.bound.sigmaN}
\min_{\lambda\in\Sigma(N(\epsilon))}\graff{|\lambda|}\geq c_1 |\epsilon|^\alpha \ .
\end{equation}
\end{proposition}

\begin{proof}
Due to the invertibility of $N$, we can rewrite $M$ as $M=N\left(\Id +
\mu T\right)$, with $T=N^{-1}P$.  We now show that $\tond{\Id + \mu
  T}^{-1}$ is well defined, thus the inverse of $M$ is given by
$M^{-1}=\left(\Id + \mu T\right)^{-1} N^{-1}$. The operator $\tond{\Id
  + \mu T}$ being a small perturbation of the identity $\Id$, its
inverse exists provided $|\mu|\opnorm{T}<1$, hence
hypothesis~\emph{(1)} implies
\begin{equation}
\label{e.small.T.1}
|\mu|\opnorm{T}\leq |\mu|\opnorm{N^{-1}}\opnorm{P}\leq
c|\epsilon|^{\beta-\alpha}<1\ ,
\end{equation}
where $c\approx c_2c_1\opnorm{P(0)}$ for sufficiently small
$\epsilon$ and in such a regime is independent of $\epsilon$; as a consequence,
since $\beta-\alpha>0$, also $M(\epsilon)$ is invertible for $\epsilon$ small
enough. From the estimate
\begin{displaymath}
\opnorm{\left(\Id + \mu T\right)^{-1}} \leq 
 \sum_{n\geq 0} \vert \mu\vert^n \opnorm{T}^n \leq
 \frac{1}{1-\vert \mu\vert \opnorm{T}} \ ,
\end{displaymath}
and recalling~\eqref{e.est.1} we get
$$
\rho(M^{-1})\leq \opnorm{M^{-1}}\leq
\frac{\opnorm{N^{-1}}}{1-\vert \mu\vert  \opnorm{T}}\ ,
$$ where the spectral radius $\rho(M^{-1})=\frac{1}{\min\lbrace |\nu_k
  |\rbrace}$, with $\nu_k\in\Sigma(M)$.  Let $\nu_1$ be the minimum, then
\begin{displaymath}
|\nu_1|\geq \frac{1-\vert \mu\vert  \opnorm{T}}{\opnorm{N^{-1}}}\ .
\end{displaymath}
For $\epsilon$ sufficiently small, condition~\eqref{e.small.T.1} holds and we have
\begin{displaymath}
|\mu|\opnorm{T}\leq \frac12\ ,\qquad\hbox{that implies}\quad |\nu_1|\geq
c_3|\epsilon|^\alpha\ ,\quad\hbox{with}\quad c_3=\frac1{2c_1}\ .
\end{displaymath}

It is clear that the main point in the proof is the bound of
$\opnorm{N^{-1}}$.  In fact, the estimate~\eqref{e.small.T.1} can be
obtained also replacing the assumption on $\opnorm{N}$ with~\eqref{e.bound.sigmaN}.  Indeed by means of~\eqref{e.est.2} one can
obtain the following estimate
\begin{displaymath}
\opnorm{N^{-1}} \leq c_{\textsf{op}}\max\graff{1,\rho(N^{-1})} =
c_{\textsf{op}}\max\graff{1,\frac1{\min_{\lambda\in\Sigma(N(\epsilon))}\graff{|\lambda|}}}
\leq c_{\textsf{op}} c_1^{-1}|\epsilon|^{-\alpha}\ ,
\end{displaymath}
since $c_1^{-1}|\epsilon|^{-\alpha}\gg 1$, which allows to conclude
\begin{displaymath}
|\nu_1|\geq c_3|\epsilon|^\alpha\ ,\qquad\hbox{with}\quad c_3= \frac12
c_{\textsf{op}}^{-1} c_1\ .
\end{displaymath}
\end{proof}

\medskip
\subsubsection{Deformation of eigenvalues.}

We aim at localizing the eigenvalues of $M(\epsilon)=N(\epsilon)+\mu(\epsilon)
P(\epsilon)$, when the spectrum of its leading part $N(\epsilon)$ is known,
provided $\epsilon$ is taken small enough in a small neighbourhood of the
origin $\Uscr(0)$. We start with a preliminary Lemma
\begin{lemma}
  \label{p.local.sp}
  Let $z\not\in\Sigma(N)$ be a complex number satisfying
  \begin{equation}
    \label{e.z.far.N}
    \dist(z,\Sigma(N))\geq 4\mu c_{\textsf{op}}\opnorm{P}\ .
  \end{equation}
  Then the following inequality holds true
  \begin{equation}
    \label{e.z.far.M}
    \frac1{c_{\textsf{op}}}\dist(z,\Sigma(N))-\mu\opnorm{P}\leq
    \dist(z,\Sigma(M))\leq c_{\textsf{op}}
    \dist(z,\Sigma(N)) + \mu c_{\textsf{op}}\opnorm{P}\ .
  \end{equation}
\end{lemma}

\begin{proof}
  To shorten the proof, let us introduce the notations
  \begin{displaymath}
    \delta_N(z)=\dist(z,\Sigma(N))\ ,\qquad
    \delta_M(z)=\dist(z,\Sigma(M))\ .
  \end{displaymath}
By setting $R_0(z)$ the resolvent of $N$, from~\eqref{e.inv.rho} we
have
\begin{displaymath}
\rho\tond{R_0(z)} = \frac1{\delta_N}\ .
\end{displaymath}
In the rest of the proof we aim at deriving bounds for $\delta_M$ by
exploiting the perturbation of $R_0$ given by $\mu P$. Thus, let us
set $R$ the resolvent of $M$ and recall that
\begin{equation*}
\Sigma(R(z)) = \graff{\frac1{\lambda_j-z}}_{\lambda_j\in\Sigma(M)}\ ;
\end{equation*}
from~\eqref{e.est.1} and~\eqref{e.est.2.bis} we have
\begin{displaymath}
\frac1{\delta_N} = \rho(R_0(z))\leq\opnorm{R_0(z)}\leq
c_{\textsf{op}}\rho(R_0(z)) = \frac{c_{\textsf{op}}}{\delta_N}\ .
\end{displaymath}
From the second Neumann series~\eqref{e.second_Neum} we can write
\begin{equation}
\label{e.R.R0}
R(z) = R_0(z)\quadr{\Id + \mu PR_0}^{-1}\ ,
\end{equation}
where, due to~\eqref{e.z.far.N}, the product $\mu PR_0$ represents a
perturbation of the identity
\begin{equation}
\label{e.bound.PR0}
  \opnorm{\mu PR_0} \leq \mu\opnorm{P}\opnorm{R_0}\leq
  \frac{c_{\textsf{op}}\mu\opnorm{P}}{\delta_N}\leq \frac14\ ,
\end{equation}
so that $\quadr{\Id + \mu PR_0}^{-1}$ is well defined in terms of
power series of $\mu PR_0$. As a consequence, from the estimate of
$\opnorm{\sum_{k\geq 0}(-1)^k\mu^k\tond{PR_0}^k}$, we get
\begin{displaymath}
\delta_M(z) = \frac1{\rho(R(z))} \geq
\frac1{\opnorm{R}}\geq \frac1{\opnorm{R_0}} - \mu\opnorm{P}\geq
\frac{\delta_N(z)}{c_{\textsf{op}}} - \mu\opnorm{P}\ ,
\end{displaymath}
which gives the lower bound in~\eqref{e.z.far.M}.
Let us consider again~\eqref{e.second_Neum} but reversing the roles of
$R$ and $R_0$
\begin{equation}
\label{e.R0.R}
R_0(z) = R(z)\quadr{\Id - \mu PR}^{-1}\ ;
\end{equation}
from~\eqref{e.R.R0} and~\eqref{e.bound.PR0} we have
\begin{displaymath}
\opnorm{R}\leq \frac43\opnorm{R_0}\qquad\hbox{that implies}\quad \opnorm{\mu PR}
\leq 
\frac13\ ,
\end{displaymath}
which provides the upper bound in~\eqref{e.z.far.M} 
\begin{displaymath}
\delta_N(z) = \frac1{\rho(R_0(z))} \geq
\frac{1}{\opnorm{R_0}}\geq \frac{3/2}{\opnorm{R}} - \mu \opnorm{P}\geq
\frac{3}{2c_{\textsf{op}}}\delta_M(z)-\mu\opnorm{P}\ .
\end{displaymath}
\end{proof}

The localization result is collected in following
\begin{proposition}
\label{p.def.eigen}
Let $M(\epsilon)\in\Mat(n)$ be decomposed into
$M(\epsilon)=N(\epsilon) + \mu(\epsilon) P(\epsilon)$. Assume that:
\begin{enumerate}[label=(\arabic*)]

\item $P(\epsilon)=P(0)+\Oscr(\epsilon)$ with $\opnorm{P(0)}\leq c_P$ and
  there exists $\beta_1>0$ such that
\begin{displaymath}
|\mu(\epsilon)|\leq |\epsilon|^{\beta_1}\ ;
\end{displaymath}

\item there exists a $c_N>0$ such that for any couple of
  distinct eigenvalues $\lambda_i\neq\lambda_j\in\Sigma(N)$
\begin{displaymath}
|\lambda_i - \lambda_j|\geq c_N
\epsilon^{\beta_2}\ ,\qquad\hbox{with}\quad\beta_2<\beta_1\ .
\end{displaymath}
\end{enumerate}
Then there exists $\epsilon^*>0$ (depending on $\Sigma(N)$) such that,
given $|\epsilon|<\epsilon^*$, for any $\lambda\in\Sigma(N)$ there exist one
eigenvalue $\nu\in\Sigma(M)$ inside the disk
$D_\epsilon(\lambda)=\graff{z\in\complessi\ :\ |z-\lambda|<c_M
  |\epsilon|^{\beta_1}}$, with $c_M>0$ a suitable constant independent of
$\lambda$.
\end{proposition}

\begin{proof}  Take an arbitrary eigenvalue $\lambda\in\Sigma(N)$ and consider a
  complex number $z\in\complessi$ at distance
  $\tilde\delta=c|\epsilon|^{\beta_1}$, with $c$ independent of $\epsilon$ to
  be determined along the proof.  In view of~\emph{(2)} and for $\epsilon$
  small enough, one has that $c_N|\epsilon|^{\beta_2}\gg
  c_P|\epsilon|^{\beta_1}$; hence by defining $\delta_N(z) =
  \dist(z,\Sigma(N))$ it turns out
\begin{displaymath}
\delta_N(z)=\tilde\delta = c|\epsilon|^{\beta_1}\ .
\end{displaymath}
We want to use upper bound of~\eqref{e.z.far.M} to control
$\delta_M(z) = \dist\tond{z,\Sigma(M)}$; to fulfil the
requirements of Lemma~\ref{p.local.sp} we take $c\geq 4c_{\textsf{op}}
c_P$ and we exploit \emph{(1)}, so that
\begin{equation}
\label{e.second.bound}
\delta_M(z) \leq
c_{\textsf{op}}(c+c_P)|\epsilon|^{\beta_1}\ .
\end{equation}
This ensures the existence of an eigenvalue $\nu\in\Sigma(M)$, which
depends on the choice of $z$, whose distance from the initially chosen
$z$ is of order $\Oscr(\epsilon^{\beta_1})$
\begin{displaymath}
\exists \nu\in\Sigma(M)\ :\ |\nu-z|\leq
c_{\textsf{op}}(c+c_P)|\epsilon|^{\beta_1}\ ,
\end{displaymath}
which provides the final estimate
\begin{displaymath}
  |\nu-\lambda|\leq |\nu-z|+|z-\lambda|\leq c_M|\epsilon|^{\beta_1}\ ,\qquad
  \hbox{with}\quad
c_M=c + c_{\textsf{op}}(c+c_P)\ .
\end{displaymath}
\end{proof}

\end{appendix}


\section*{Acknowledgments}
We warmly thank B.~Langella for her help on the theory of
resolvent. We feel the lack of our friend Massimo Tarallo, who has
helped us with enlightening discussions on matrix norms and spectral
radius. M.S., T.P. and V.D. were partially supported by the
National Group of Mathematical Physics (GNFM-INdAM) and by the
MIUR-PRIN 20178CJA2B ``New Frontiers of Celestial Mechanics: theory
and Applications''.


\def\cprime{$'$} \def\i{\ii}\def\cprime{$'$} \def\cprime{$'$}

\end{document}